\documentclass[11pt,a4paper,reqno,oneside]{amsart}
\usepackage{vmargin,color}
\usepackage[latin1]{inputenc}
\usepackage{amsmath,amsfonts,amsthm,epsfig,graphicx}
\usepackage{caption}
\usepackage{ esint }

\def\H{\mathcal{H}}

\def\N{\mathbb N}

\def\R{\mathbb R}

\def\H{\mathcal{H}^{n-1}}

\def\e{\varepsilon}

\newtheorem{theorem}{Theorem}[section]
\newtheorem{definition}[theorem]{Definition}

\newtheorem{proposition}[theorem]{Proposition}
\newtheorem{lemma}[theorem]{Lemma}
\newtheorem{corollary}[theorem]{Corollary}

\theoremstyle{remark}
\newtheorem{remark}{Remark}[section]

\numberwithin{equation}{section}
\numberwithin{figure}{section}

\address{Department of Mathematics, The University of Texas at Austin, 2515 Speedway, Stop C1200, Austin TX 78712-1202, USA}
\email{caffarel@math.utexas.edu}
\email{antoniofarah@utexas.edu}
\address{Department of Mathematics, Johns Hopkins University, 3400 N. Charles Street, Baltimore, MD 21218, United States of America}
\email{drestre1@jh.edu}

\title{On the Grad-Mercier equation and Semilinear Free Boundary Problems}
\author{Luis Caffarelli, Antonio Farah, and Daniel Restrepo}

\begin{document}
	
	\begin{abstract} In this paper, we establish regularity and uniqueness results for Grad-Mercier type equations that arise in the context of plasma physics. We show that solutions of this problem naturally develop a dead core, which corresponds to the set where the solutions become identically equal to their maximum. We prove uniqueness, sharp regularity, and non-degeneracy bounds for solutions under suitable assumptions on the reaction term. Of independent interest, our methods allow us to prove that the free boundaries of a broad class of semilinear equations have locally finite $\H$ measure. 
	\end{abstract}

	\maketitle
\bigskip

\section{Introduction}

\subsection{Overview} This paper is devoted to the study of the boundary value problem
\begin{equation}
	\label{main problem grad}
	\begin{cases}
		-\Delta u(x) = g(|u \geq u(x)|),\quad \text{in}\, \, \Omega,\\
		u=0, \qquad \text{on}\, \, \partial \Omega,
	\end{cases}
\end{equation}
where, on the right hand side, we use the notation
$$u\geq u(x):=\{y\in \Omega:u(y)\geq u(x)\}$$
for the superlevel sets of $u$, and $\vert\cdot\vert$ to denote the $n$-dimensional Lebesgue measure. Furthermore, $\Omega \subset \R^n$ is an open, bounded, and connected set with $C^{1,1}$ boundary, and $g:[0,|\Omega|]\to \R$ is a Lipschitz continuous function.

Equations of the form \eqref{main problem grad} are known in the literature as queer differential equations (QDEs) or as Grad-Mercier equations. They were originally introduced by Harold Grad in the context of nuclear fusion, more precisely as a model for adiabatic compression \cite{grad1975adiabatic} and resistive diffusion \cite{grad1979alternating} of plasma in toroidal or spherical containers (TOKAMAK).

There are many equations associated with the modeling of plasma physics that have been widely studied in the literature of elliptic PDE (see, e.g., \cite{caffarelli1982convexity,kindspruck,mt}). All of these approaches depart from the magneto-hydrodynamic equilibrium condition
\begin{equation}\label{eq ehd condition}
	-\Delta u = \frac{d p(u)}{du},
\end{equation} 
where $u$ stands for the plasma flux function and $p$ is the (given) pressure function which depends non-linearly on $u$. Even in its most basic and standard formulations, equation \eqref{eq ehd condition} is coupled with some extra conditions that lead to widely studied free boundary problems (see the survey \cite{t}). The presence of the free boundary is physically explained by the segregation of the plasma in two domains, where in one of them the flux is described by \eqref{eq ehd condition}, and, where in the other, which corresponds to a surrounding vacuum region, $u$ is conserved (e.g., harmonic). \\

Equations of the form \eqref{main problem grad} arise when we allow $p$ to depend non-locally on $u$. More precisely, imposing adiabatic constraints in the plasma flux, \cite{grad1975adiabatic} proposed a pressure profile of the form $p(u) = \mu(u) ({u^*}')^\gamma$, where $\gamma \in \R$, $\mu:\R\to \R$ is a smooth function, $u^*$ is the classical decreasing symmetric rearrangement of $u$, and ${u^*}'$ represents the derivative defined a.e. of such a monotonic real-valued function. In this case, $g$ in \eqref{main problem grad} depends on $u$ in a very non-standard and delicate way. Rigorous results in this direction were pioneered by Laurence and Stredulinsky; in a series of works initiated in   \cite{laurence1985new}, they proposed a scheme to obtain solutions to the original Grad equation with suitable geometric features (like convexity on their level sets). This programme ended in \cite{laurence1996gradient}, where they were able to prove existence of minimizers (in $\R^2$) with suitable gradient bounds. Aside of this, major problems like uniqueness, existence in any dimension, and regularity theory for Grad equations are widely open.\\

In this paper, we opt for the approach introduced in the works of Temam and Mossino \cite{mossino1978application, mossino1982priori}, where $g$ is assumed to be a given continuous function. Given the serious technical difficulties to obtain further conclusions from the original equation proposed by Grad, it is still reasonable to address the non-local problem \eqref{main problem grad} to gain insight into what is expected from more complicated models. Furthermore, even in this simplified framework, there are still open problems and technical challenges, such as uniqueness and regularity of solutions, which were posed more than 40 years ago in expository works \cite{t,grad1977magnetic}. Part of this work is devoted to the resolution of some of these open problems.\\

Additionally, we introduce a novel free boundary problem overlooked so far in the study of this problem. We show that, under certain assumptions on $g$, a solution $u$ of \eqref{main problem grad} has a plateau at its maximum level with prescribed Lebesgue measure (depending solely on $g$), i.e., the solution develops a free boundary at $FB(u) := \partial \{x\in \Omega\, |\, u(x) = \max u\}$. After considering the transformation $v=\max u -u$, \eqref{main problem grad} can be written as a semilinear free boundary problem of the form
\begin{equation}\label{eq gen sem}
    \begin{cases}
        \Delta v= f(v), \hspace{2mm} \quad \{v>0\},\\
        |\nabla v| =0, \qquad \partial \{v>0\},
    \end{cases}
\end{equation}
where $f$ is increasing. Under very general hypotheses on $g$, it is possible to show $f(t)\geq \frac{1}{C}t^\frac{1}{3}$, implying that the solution to \eqref{eq gen sem} is also a subsolution to the Alt-Phillips equation 
\begin{equation}\label{eq alt phillips}
    \begin{cases}
        \Delta \rho = C_0 \rho^{\gamma-1}, \hspace{2mm} \quad \{\rho>0\},\\
        |\nabla \rho| =0, \qquad \partial \{\rho>0\},
    \end{cases}
\end{equation}
for some $C_0$ and $\gamma = \frac{4}{3}$ (see \cite{alt1986free}). On the other hand, based on the sharp estimates obtained in the radially symmetric case, it is expected that in general $f(t)\leq C t^\frac{1}{3}$ (see Corollary \ref{corollary radial}), but unfortunately our methods only yield the suboptimal bound $f(t)\leq C t^\frac{1}{5}$. This conjecture is also reinforced by the fact that the finite $n-1$ dimensional measure of the free boundary $\partial \{ v>0\}$ is closely related to the optimal bound $f(t)\leq C t^\frac{1}{3}$ (see Theorem \ref{thm equivalence}).\\

Lastly, we obtain a result of independent interest; namely, we find conditions on the right hand side of the general semilinear free boundary problem \eqref{eq gen sem} that guarantee the $\mathcal{H}^{n-1}$-finiteness of the free boundary of solutions.

\subsection{Statements of the main results}

We begin by establishing our conventions.\\

\noindent
\textbf{Assumptions:} Besides Lipschitz continuity of $g$, we will require one of the following two properties:
\begin{enumerate}
	\item [$(H_1)$] $g$ is strictly positive.
	\item [$(H_2)$] $g(0)\leq 0$ and there exists a unique $\alpha \in [0, |\Omega|)$ such that $g(\alpha) = 0$.
\end{enumerate}

\medskip

\noindent
\textbf{Notations:} Given a function $w \in C(\overline{\Omega})$, we denote its maximum in $\Omega$ by $\max w$ and, accordingly, we consider the set $D_w := \{x \in \Omega \, |\, w(x)= \max w\}$ where $w$ attains its maximum. Given any $t \in \R$, we denote the superlevel set of $u$ by $u\geq t:=\{y\in \Omega:u(y)\geq t\}$; we usually write the Lebesgue measure of this set simply as $|w\geq t|$.
Finally, we say that a real number is a \textbf{universal constant} if it is positive and can be defined in terms of $g$ and $\Omega$ only. Following a widely used convention, we will use the letter $C$ for a generically ``large'' universal constant, and $1/C$ for a generically ``small'' one.\\

\medskip

\noindent 
\textbf{Main results:} Our first result shows the uniqueness of solutions to \eqref{main problem grad}, solving one the open problems posed in \cite{t}.

\begin{theorem}\label{thm 1}
	Let us assume that $g$ is non-decreasing. Then, the problem \eqref{main problem grad} admits exactly one solution $u$. Furthermore, this solution is $C^{1,1}$ up to the boundary and has continuous Laplacian.
\end{theorem}

We stress that our argument for proving uniqueness is non-variational and also applies for other elliptic operators as those considered in \cite{caffarelli2023fully} (see Section 3).\\

Our second result introduces the free boundary for \eqref{main problem grad} and describes the behavior of $u$ near $\partial D_u$.

\begin{theorem}\label{thm 0}
	Let us assume that $g$ satisfies $(H_2)$ with $\alpha>0$, then  $|D_u| = \alpha$. If, furthermore, $g'(t)\geq \frac{1}{C}$ for a.e. $t \in [\alpha, |\Omega|]$, then the solution $u$ to \eqref{main problem grad} satisfies
    \begin{equation}\label{eq upper g}
	g(|u \geq u(x)|) \leq C(\max u-u)^\frac{1}{5} \quad \mbox{in $\Omega$,}
    \end{equation} 
    and
    	\begin{equation}\label{eq lower g}
		\frac{1}{C}(\max u-u)^\frac{1}{3}\leq g(|u \geq u(x)|) \quad \mbox{in $\Omega$.}
	\end{equation} 
	
\end{theorem}

The solutions to \eqref{main problem grad} are expected to belong to $C^{2,1}(\Omega)$ (in the absence of critical points outside the dead core). This translates into cubic detachment of solutions from their free boundary. In this regard, the inequality \eqref{eq lower g} is sharp and corresponds to the non-degeneracy of the solutions near their free boundary (see Section \ref{sec: fb}), while \eqref{eq upper g} is expected to be suboptimal. A relevant corollary of Theorem \ref{thm 1}, supporting the previous observation, is the complete resolution of \eqref{main problem grad} when $\Omega$ is radially symmetric. In this case, we are also able to show that the best possible regularity for solutions is $C^{2,1}$ (up to the free boundary).

\begin{corollary}\label{corollary radial}
	Let $g$ be a nondecreasing, smooth function satisfying $(H_2)$ with $\alpha >0$ and such that $g'(t)\geq \frac{1}{C}$ for $t \in [\alpha, |\Omega|]$. If $\Omega$ is radially symmetric, then $u$ is radially symmetric and its deadcore $D_u$ is a homothety of $\Omega$ with measure $\alpha$.\\
    Furthermore, $u\in C^{2,1}(\Omega)$,
    \begin{equation}\label{eq optimal det}
       \frac{1}{C} \text{dist}(x, N_u)^3\leq  \max u -u \leq C \text{dist}(x, N_u)^3,
    \end{equation}
and
\begin{equation}\label{eq AP comparison}
           \frac{1}{C}(\max u-u)^\frac{1}{3} \leq g(|u \geq u(x)|) \leq C(\max u-u)^\frac{1}{3}.
        \end{equation}
\end{corollary}

\medskip

In the previous result, the nonlinearity $g$ is assumed to be smooth simply to highlight the sharp regularity of a solution to \eqref{main problem grad} near its free boundary.\\

In free boundary problems, it is often the case that the finite perimeter of the free boundary follows from the optimal regularity of solutions (among other properties); in this case they are in fact somewhat equivalent.

\begin{theorem}\label{thm equivalence}
Assume that $g$ satisfies $(H_2)$ with $\alpha>0$. If $g(|\max u -u \leq t|)  \leq Ct^{\frac{1}{3}}$, then $\H(FB(u)) < \infty$.\\

Conversely, if there exists a sequence of real numbers $\{t_k\}$ converging to $0$ such that $\mathcal{H}^{n-1}(u = \max u - t_k) \leq C$, then $g(| \max u - u \leq t|) \leq Ct^{\frac{1}{3}}$.
\end{theorem}

Finally, we provide sufficient conditions on $f$ so that solutions to a semilinear equation of the form $\Delta v = f(v)$ always develop free boundaries with a finite  $\mathcal{H}^{n-1}$ measure. We achieve this by transforming the problem into a degenerate one-phase problem like those studied by De Silva and Savin in \cite{de2021certain} and for that purpose, we require (essentially) that the right-hand side of the equation does not oscillate too much. More precisely, given $f: \R_+\to \R$ continuous and strictly increasing, we consider the functions $F(t) = \int_{0} ^ {t} f(s) ds$, and 
\begin{equation}\label{eq transform}
    h(t) = \int_{0} ^ {t} \frac{1}{\sqrt{F(s)}} ds,
\end{equation} 
and the following set of hypotheses:
\begin{enumerate}
    \item [$(A_1)$:] $h(t) < \infty$ for $t\geq 0$. 
    \item [$(A_2)$:] $\lim_{t \rightarrow 0^{+}} \frac{tf(t)}{F(t)} = \omega \in [1,2)$
    \item [$(A_3)$:] $f(t) \leq t^{\epsilon}$ for some $\epsilon > 0$.
\end{enumerate}
\begin{theorem}\label{thm conditional}
Consider the problem 
\begin{equation}\label{eq semilinear 0}
\begin{cases}
    \Delta v = f(v) \qquad \text{ in } \{v>0\}\\
|\nabla v| = 0 \qquad \text{ on } FB(v) := \partial \{v>0\},
\end{cases}
\end{equation}
with $f$ strictly increasing, continuous, $f(0) = 0$, and satisfying $(A_1)-(A_3)$.\\ 

Then, $\H(FB(v) \cap B) < \infty$ for any ball $B$.
\end{theorem}

As a corollary, we obtain the optimal growth bound for the right-hand side of the Grad equation under the extra assumption $(A_2)$
\begin{corollary}\label{corollary grad cond}
Assume that $g$ satisfies the hypothesis of Theorem \ref{thm 0}. Setting $f(t) = g(|\max u - u \leq t|)$, if $f$ satisfies assumption $(A_2)$ from Theorem \ref{thm conditional}, then 
\begin{equation}
f(t) \leq Ct^{\frac{1}{3}}.    
\end{equation}
\end{corollary}

\medskip

\subsection{Remarks on Theorem \ref{thm conditional}}  In the absence of scale invariances, the analysis of free boundaries for solutions to semilinear equations is substantially more challenging than in benchmark problems such as the Alt-Phillips equation \eqref{eq alt phillips}. In order to circumvent this challenge, in the proof of Theorem \ref{thm conditional} we transform \eqref{eq semilinear 0} into a general one-phase problem using the transformation $w=h(v)$ with $h$ given by \eqref{eq transform}.  The function $h$, which is also strictly monotone, appears naturally in the analysis of the one dimensional version of \eqref{eq reversed problem}, since $U(t) = h^{-1}(\sqrt{2}t)$ solves
 \begin{equation}\label{eq one dimensional}
     \begin{cases}
         U''(t) = f\big(U(t)\big),\\
         U(0)=0.
     \end{cases}
 \end{equation}
Indeed, if we multiply both sides of the equation in \eqref{eq one dimensional} by $U'$, we deduce from the fundamental theorem of calculus that $(U')^2= 2 F(U)$ or, equivalently, $\big(h(U)'\big)^2=2$, from where we can see that our choice of $U$ satisfies \eqref{eq one dimensional}. More fundamentally, the finiteness of $h$, i.e., assumption $(A_1)$ is a necessary and sufficient condition to the existence of free boundaries for semilinear equations \cite[Theorem 1.1.1]{pucci2007maximum}. Moreover, the problem satisfied by $w$ takes the form
\begin{equation}\label{eq one phase}
\begin{cases}
        \Delta w(x) = a(x)\Bigg(\frac{2-|\nabla w(x)|^2}{w}\Bigg), \quad \mbox{in $\{w >0\}$}\\
        |\nabla w|^2=2, \hspace{3.1cm} \mbox{on $\partial \{w >0\}$},
\end{cases}
\end{equation}
where $a(x) =  \frac{f(v(x)) h(v(x))}{2\sqrt{F}(v(x))}$. This equation is similar to those treated by De Silva and Savin in \cite{de2021certain}, but with a less regular right hand side that introduces some technical challenges, which we are able to address by introducing the hypotheses $(A_2)-(A_3)$.\\ 

Up to the best of our knowledge, Theorem \ref{thm conditional} is the first result that guarantees finiteness of the perimeter for free boundaries of semilinear equations in the absence of scale invariances. An important consequence of this result is that, thanks to De Giorgi structure theorem for sets of finite perimeter, such free boundaries always have points with an almost tangent plane or, in the jargon of free boundary problems, ``flat points''. This, coupled with the recent results in \cite{allen2024free}, shows that for a broad class of semilinear free boundary problems, any solution $v$ satisfies that the reduced boundary $\partial^* \{v>0\}$  is smooth.\\

Examples of functions satisfying the hypotheses of Theorem \ref{thm conditional} and those in \cite{allen2024free} include
\begin{itemize}
    \item $f(t) = t^{\gamma}(2 + \sin(t))$
    \item $f(t) = -t^{\gamma}\log(t)$
    \item $f(t) = t^{\gamma}e^{-t}$
\end{itemize}
where $\gamma \in (0,1)$.
\subsection{Organization of the paper.} Section \ref{section def and pre} is devoted to gathering previous results from the literature and to unify them in a simplified framework. In Section \ref{sec: uniqueness}, we prove Theorem \ref{thm 1} and Corollary \ref{corollary radial}. Finally, in Section \ref{sec: fb}  we introduce the free boundary associated with \ref{main problem grad} and develop our analysis of the free boundary, proving Theorem \ref{thm 0}, Theorem \ref{thm equivalence}, Theorem \ref{thm conditional}, and Corollary \ref{corollary grad cond}. \\

\medskip

\noindent {\textbf Acknowledgments:} We thank Ignacio Tomasetti for his collaboration on a preliminary stage of this project. Author Antonio Farah wishes to thank his advisors Luis Caffarelli and Irene Gamba for their continued support and encouragement. Author Luis Caffarelli was partially supported by NSF Grant DMS 2000041 and NSF Grant DMS 2108682. Author Antonio Farah was partially supported by NSF Grant DMS 1840314 and NSF Grant DMS 2009736.

\section{Preliminary Results}\label{section def and pre}

We start by unifying different notions of solutions for \eqref{main problem grad}. We recall that the existence of weak solutions for \eqref{main problem grad} was obtained in \cite{mt}, later generalized in \cite{caffarelli2023fully} for fully nonlinear operators in the context of $W^{2,p}$ viscosity solutions. For the convenience of the reader, we recall the latter notion of solution introduced in \cite{caffarelli1996viscosity}.

	\begin{definition}\label{viscositydefinition}
	Let $f\in L^{p}(\Omega)$ and $\psi\in W^{2,p}(\Omega)$ for $p>n$. Let $u:\Omega\longrightarrow\mathbb{R}$ be a continuous function. We say that $u$ is a 						\textbf{$W^{2,p}$-viscosity subsolution} of 
		\begin{equation}\label{viscosityequation}
		-\Delta u =f(x)
		\end{equation} 
	in $\Omega$, if $u\leq 0$ on $\partial\Omega$ and the following holds: for all $\varphi\in W^{2,p}(\Omega)$ such that $u-\varphi$ has a local maximum at $x_0\in \Omega$, then 
	$$ess\limsup_{x\rightarrow x_0}\Delta \varphi +f(x)\geq 0.$$
	In the same way, we define supersolutions. We say that $u$ is a \textbf{$W^{2,p}$-viscosity supersolution} of \eqref{viscosityequation} in $\Omega$, if $u\geq 0$ on $\partial\Omega$ 	and the following holds: for all $\varphi\in W^{2,p}(\Omega)$ such that $u-\varphi$ has a local minimum at $x_0\in \Omega$ then 
	$$ess\liminf_{x\rightarrow x_0} \Delta \varphi+f(x)\leq 0.$$
	We will call $u$ a \textbf{$W^{2,p}$-viscosity solution} if it is both a subsolution and a supersolution.
	\end{definition}

We note that thanks to \cite[Lemma 2.6]{caffarelli1996viscosity} and \cite[Corollary 3.7]{caffarelli1996viscosity}, the notion of  $W^{2,p}$-viscosity solutions and strong solutions coincide for $W^{2,p}$ functions. In the following lemma, we show the existence of dead cores under assumption $(H_2)$ and, in general, the continuity of the right-hand side in \eqref{main problem grad}, which, in turn, implies that solutions to \eqref{main problem grad} are in fact standard $C$-viscosity solutions (like those studied in \cite{caffacabre}).
\begin{lemma}\label{lemma dead core}
     Let $u$ be a $W^{2,p}$-viscosity solution to \eqref{main problem grad}. If $g$ satisfies $(H_1)$ or $(H_2)$, then $ t\to g(|u\geq t|) $ is continuous in $[0, \max u]$. Furthermore, under hypothesis $(H_2)$, $|u=\max u|=\alpha$.
\end{lemma}
\begin{proof}
    Let us first notice that if $u \in W^{2,p}(\Omega)$ and $-\Delta u >0$ a.e. in $\{u=t\}$ for some $t \in \R$, then $|u=t|=0$. Indeed, if $|u=t|>0$, by a theorem of Stampacchia (see \cite[Remark 3.1]{caffarelli2023fully})  $D^2 u =0$ a.e. in $\{u=t\}$, implying that $f(|u\geq t|)=-\Delta u = 0$ a.e. in $\{u=t\}$, since $u$ is a strong solution of \eqref{main problem grad}. From this observation, we readily deduce that $t\to |u\geq t|$ is continuous in $[0, \max u]$ under hypothesis $(H_1)$.\\
	
	If instead $u$ is a solution of \eqref{main problem grad} under hypothesis $(H_2)$, we claim that $|u \geq \max u|\geq \alpha$. If $\alpha=0$, this claim is trivial. We thus assume $\alpha>0$. Since $u$ is a $W^{2,p}$-viscosity solution for $p>n$ (and therefore continuous) and is equal to zero on $\partial \Omega$, $u$ attains its  maximum at a certain point $x_{0} \in \Omega$. Hence, by setting $\varphi\equiv u_{max}$ in Definition \ref{viscositydefinition}, we have that
	\begin{equation*}\label{eq limit}
		0\leq ess\limsup_{x\rightarrow x_0}\Delta \varphi(x))+g(|u\geq  u(x)|)=g(|u\geq  u(x_0)|),
	\end{equation*}
	where the rightmost equality is justified by the continuity of $g$, $u$, and the continuity from the left of $|u \geq  t|$ at $\max u$. Indeed, let $t_n \nearrow \max u=u(x_0)$. Then, the continuity properties of Lebesgue measure imply $|u= \max u| = \left| \bigcap_{n\in \N} \{u\geq t_n\}\right|=\lim_{n\to \infty} |u\geq t_n|$.\\
    
    On the other hand, since $\{u= \max u\}$ has positive measure, we have that $D^2 u =0$ a.e. in $\{u= \max u\}$, therefore $0= -\Delta u(x) = g(|u\geq  u(x)|)$ a.e. in $\{u = \max u\}$. So, thanks to hypothesis $(H_2)$, we deduce that $|u\geq  u(x)| = \alpha$.\\

    By the previous analysis, under $(H_2)$ we have that $g(|u\geq t|)>0$ for $t \in [0, \max u)$. So, by our previous arguments, it is clear that $t\to |u\geq t|$ is continuous on $[0,\max u )$, finishing the proof.        
\end{proof}

\begin{corollary}\label{corollary cvisco}
     Let $u$ be a $W^{2,p}$-viscosity solution of \eqref{main problem grad}. Then, $u$ is a $C$-viscosity solution of the same problem.
\end{corollary}
\begin{proof}
In virtue of Lemma \ref{lemma dead core}, the right-hand side of \eqref{main problem grad} is continuous. Then, by testing with any $\varphi \in C^2$ in the definition of $W^{2,p}$, we retrieve the usual $C$-viscosity condition.
\end{proof}

The next lemma will allow us to apply directly the classical maximum principle in our context.
\begin{lemma}\label{lemma to avoid viscosity definition}
	Let $u_i\in W^{2,p}(O)$ be $W^{2,p}$-viscosity solutions for $-\Delta u=f_i$ in an open set $O$ for $i=1,2$ and with $f_1$ and $f_2$ continuous in $O$. If $u_1-u_2$ attains a local maximum (minimum) at $x_0$, then $f_1(x_0)\geq f_2(x_0)$ ($f_1(x_0)\leq f_2(x_0)$). 
\end{lemma}

\begin{proof}
	Since $u_i \in W^{2,p}(O)$, $-\Delta  u_i=f_i$ a.e. in $O$, for $i=1,2$. By Definition \ref{viscositydefinition}, we can use $u_2$ as a $W^{2,p}$ viscosity test for $u_1$ which is, in particular, a viscosity subsolution of $-\Delta u_1 = f_1$. Then, we deduce from the continuity of $f_1$ and $f_2$ that
	
	$$0\leq ess\limsup_{x\rightarrow x_0}-\Delta u_2(x)+f_1(x)=-f_2(x_0)+f_1(x_0).$$	
	
	The case when $u_1-u_2$ attains a local minimum is analogous.
\end{proof}	

From the existence results in \cite{mt} and \cite{caffarelli2023fully}, we obtain for free the $W^{2,p}$ regularity of solutions for any $p\geq 1$ and, by Morrey's embedding, $C^{1,\alpha}$ for all $\alpha \in (0,1)$. This regularity can be improved (again for free) if we assume monotonicity on $g$ by invoking the main result in \cite{shahgholian2003c1}.

\begin{lemma}\label{lemma regularity Shagohlian}
Let $u$ be a solution to \eqref{main problem grad}. If we assume $g$ to be non-decreasing, then $u \in C^{1,1}(\Omega)$.
\end{lemma}
\begin{proof}
Notice that $v = \max u - u$ solves
\begin{equation*}
\Delta v = f(v)
\end{equation*}
with $f(t) = g(|v\leq t|)$ a non-decreasing and bounded function. Therefore, we are under the hypothesis of \cite[Theorem 1.1]{shahgholian2003c1} which provides the desired regularity.
\end{proof}

Further regularity for solutions to \eqref{main problem grad} is tightly linked with the distribution of their critical points and, up to the best of our knowledge, is unknown. Actually, it is still an open problem to determine conditions to guarantee $C^2$ regularity of solutions to the semilinear problems of the form $\Delta v = f(v)$ when $f$ is continuous. This has been accomplished for instance in \cite{koch2015partial} at regular points of $v$. In the case of solutions to \eqref{main problem grad}, the interplay between critical points and regularity is more transparent, as we show in the following result.
\begin{lemma}\label{lemma boundary regularity}
	Let $u$ be a Lipschitz function and let $r_1< r_2$ such that $\mathcal{H}^{n-1}(\{u=t\})\leq C_1$ for $t \in [r_1,r_2]$ and such that $|\nabla u|\geq C_2$ in $\{r_1\leq u \leq r_2\}$. Then, $h(t)=|u\geq t|$ is a Lipschitz function in $[r_1,r_2]$. Moreover, if $u \in C^{k,\alpha}$ for $k\in \N$ and $\alpha \in [0,1]$, then $h(t)\in C^{k-1,\alpha}([r_1,r_2])$.  In particular, if $\Omega$ is a $C^{1,1}$ domain and if $u$ is a solution of \eqref{main problem grad}, then, the function $x\to |u\geq u(x)|$ is Lipschitz in a $\delta$-neighborhood of $\partial \Omega$. 
\end{lemma}
\begin{proof}	
	By the coarea formula, if $t,s\in [r_1,r_2]$ with $s>t$	
	$$|u\geq t|-|u\geq s| =\int_{\{s> u\geq t\}} \frac{|\nabla u(x)|}{|\nabla u(x)|}dx= \int_{t}^s \int_{\{u=r\}} \frac{1}{|\nabla u(x)|}d\mathcal{H}^{n-1}(x)dr\leq |s-t| \frac{C_1}{C_2}.$$
Moreover, 
\begin{equation*}
    h'(t) = - \int_{\{u=t\}} \frac{1}{|\nabla u(x)|}d\mathcal{H}^{n-1}(x)
\end{equation*}
for almost every $t\in [r_1,r_2]$.\\

If  $u \in C^{k,\alpha}$, we can invoke the main result in \cite{l}, which asserts that $h'$ is as regular as $\nabla u$, i.e., $C^{k-1,\alpha}$.\\

If $u$ solves \eqref{main problem grad}, we can combine Hopf's boundary lemma and the $C^{1,\alpha}$ regularity of $u$ to conclude that $|\nabla u|\geq c$ in some $\delta$ neighborhood of $\delta \Omega$, $I_{\delta}(\partial \Omega)$ containing $\{0\leq u \leq r_1\}$ for some $r_1>0$.  The 		regular value theorem guarantees that $\Omega_t = \Omega\setminus \{u\geq t\}$ is a $C^\alpha$ domain for $t\in (0,r_1]$. Applying the divergence theorem on these sets, we deduce that 
	$$ \int_{\Omega_t} -\Delta u = -\int_{\{u=t\}} |\nabla u|d\mathcal{H}^{n-1}-\int_{\partial\Omega} \frac{\partial u}{\partial \nu}d\mathcal{H}^{n-1}.$$
	Hence,
	$$\mathcal{H}^{n-1}(\{u=t\})\leq \left\Vert  \frac{\partial u}{\partial \nu}\right\Vert_{L^\infty(\partial \Omega)}\frac{1}{C}  \mathcal{H}^{n-1}({\partial\Omega})+C\Vert u \Vert_{W^{2,1}(\Omega)}.$$
\end{proof}

Assuming $g$ is smooth, the previous lemma, combined with standard elliptic regularity, provides smoothness for $u$ and $t\to |u\geq t|$, as long as $u$ does not have critical points. In particular, this is always the case near $\partial \Omega$. An interesting question that we leave as an open problem is the analyticity of solutions under these circumstances (assuming analyticity of $g$ as well).\\

\begin{lemma}\label{lemma abs cont}
    Let $u$ be a solution of \eqref{main problem grad} with $g$ satisfying either $(H_1)$ or $(H_2)$. Then, $h(t) = |u\geq t|$ is absolutely continuous and a.e. $t \in (0, \max u)$
    \begin{equation}\label{eq der reaction term}
        h'(t) = - \int_{\{u=t\}} \frac{1}{|\nabla u|} d\mathcal{H}^{n-1}.
    \end{equation}
\end{lemma}
\begin{proof}
    Since $g(|u \geq t|) >0$ for $t \in (0, \max u)$, we have that $| \{\nabla u =0\} \cap \{0<u<\max u\}|=0$, otherwise we would have $0=-\Delta u =f(u) >0$ in $\{\nabla u =0\} \cap \{0<u<\max u\}$. Thus, by the coarea formula, for $0<s<t<\max u$
    \begin{equation*}
        |s\leq u \leq t| = |\{s\leq u \leq t\} \cap \{\nabla u \neq 0\}| = \int_{s}^{t}\int_{\{u=\kappa \} \cap \{\nabla u \neq 0\}} \frac{1}{|\nabla u|} d\mathcal{H}^{n-1}d\kappa.
    \end{equation*}

    Another (very classical) application of the coarea formula yields
    \begin{equation*}
      0=  \int_{\{\nabla u=0\}\cap \{0<u<\max u\}} |\nabla u| = \int_0^{\max u} \mathcal{H}^{n-1}\big(\{\nabla u=0\}\cap \{ u=t\}\big) dt,
    \end{equation*}
    implying that $\mathcal{H}^{n-1}\big(\{\nabla u=0\}\cap \{ u=t\}\big)=0$ for $t$ a.e. in $(0, \max u)$, which concludes the derivation of \eqref{eq der reaction term}.
\end{proof}

\section{Uniqueness and comparison principles}\label{sec: uniqueness}

In this section, we prove the uniqueness of solutions of \eqref{main problem grad}. Strictly speaking, the uniqueness result holds under very mild assumptions: it requires $\Omega$ to satisfy the uniform interior sphere condition and $g$ to be increasing and Lipschitz in an interval of the form $[|\Omega|-\delta, |\Omega|]$. Additionally, we can replace the Laplace operator in \eqref{main problem grad} by a convex elliptic Fully Nonlinear operator $F$  as in \cite{caffarelli2023fully} satisfying the following condition: there exists a non-increasing function $c:(0,1] \to (0,1]$ with $c(t)<1$ in $(0,t)$ such that $F(tM)\leq c(t)F(M)$ whenever $F(M)\leq 0$. This previous condition is fulfilled, for instance, if $F(t M)=t^\beta F(M)$ for any matrix $A$ and for some $\beta>0$. Thus, this uniqueness argument allows us to recover many relevant operators such as linear operators, supremums or infimums of linear operators (such as Isaacs equation), the $p$-Laplacian for any $p\in (1,\infty)$, Pucci operators, the Monge-Ampere operator and any symmetric polynomial on the eigenvalues of $M$.\\

The main idea in the following result is a variant of Krasnoselkii's method introduced in \cite{bk}.

\begin{theorem}\label{uniqueness}
	Let $\Omega$ be a domain satisfying the interior sphere condition and let $g$ be non-decreasing and Lipschitz in the set where $g\geq 0$. If, furthermore, $g$ satisfies $(H_1)$ or $(H_2)$, \eqref{main problem grad} has at most one solution.
\end{theorem}

\begin{proof}
	Let $u_1$ and $u_2$ be two solutions to \eqref{main problem grad}. Consider the set 	
	$$\Lambda =\{t\in [0,1]\, : u_2\geq tu_1\},$$	
	and define $t_0 =  \sup_{t\in\Lambda}t$.\\
	
	\medskip

    \noindent
	\textit{Step 1} We show that $t_0>0$.\\
	
	Both $u_1$ and $u_2$ are strictly positive and are viscosity supersolutions of $\Delta  w= 0$, i.e., superharmonic. Then both satisfy Hopf's lemma. Combining this with the $C^{1,1}$ regularity up to the boundary, we deduce that $\Lambda$ contains a neighborhood of $0$. \\
	
	\medskip

    \noindent	
	\textit{Step 2} We prove that under $(H_1)$, $t_0\geq 1$; while, under $(H_2)$ the following dichotomy holds
	\begin{enumerate}
		\item $t_0 \geq 1$,
		\item or $\max u_2 = t_0 \max u_1$ and, moreover, $u_2-t_0u_1$ is zero at interior points where both $u_1$ and $u_2$ attain simultaneously their maximum value.
	\end{enumerate}

	Our goal in either case is to show the existence of $y_0 \in \Omega$ such that $u_2(y_0) = t_0 u_1(y_0)$. By the definition of $t_0$, it suffices to show that $\frac{\partial}{\partial \nu}(u_2-t_0 u_1)(x)<0$ on $\partial \Omega$, where $\nu$ is the outer unit normal to $\partial \Omega$. With this purpose in mind, we will show that $u_2-t_0 u_1$ is a supersolution of a suitable operator that satisfies Hopf's maximum principle. \\
	
	To begin, notice that by combining the Lischitz regularity of $g$ with Lemma \ref{lemma boundary regularity} we can find $\varepsilon>0$ such that the functions $f_i(s) := g(|u_i\geq s|)$ are Lipschitz continuous in $[0, \varepsilon]$, for $i= 1,2$. Therefore, there exists $M>0$ such that $f_2(s)+Ms$ is strictly increasing in $[0,\varepsilon]$. Let us take $\delta>0$ small enough such that $ I_\delta(\Omega) \subset \{u_2 \leq \varepsilon\}$. So, for $x \in I_\delta (\Omega)$ and for $t_0\in (0,1]$, we have that
	\begin{eqnarray}\nonumber
		 \Delta (t_0u_1- u_2)(x)-M(t_0u_1-u_2)(x)&=&g(|u_2\geq u_2(x)|) -t_0 g(|u_1\geq u_1(x)|)\\ \nonumber
        && -M(t_0u_1-u_2)(x)\\ \nonumber
		&\geq& g(|u_2\geq u_2(x)|)-t_0 g(|u_2\geq t_0u_1(x)|)\\ \nonumber
        &&-M(t_0u_1-u_2)(x)\\ \nonumber
		&= & f_2(u_2(x)) - f_2(t_0u_1(x))\\ \label{eq hopf lemma}
		&\geq& 0,
	\end{eqnarray}
	where in the first inequality we use the monotonicity of $g$ together with the fact that $u_2\geq t_0 u_1$ in $\Omega$, so $\{u_1\geq u_1(x)\}\subset\{u_2\geq t_0 u_1(x)\}$, while in the second inequality, we just use the monotonicity of $f_2$. Thus, Hopf's lemma implies that $\frac{\partial}{\partial \nu}(u_2-t_0 u_1)(x)<0$ at any point of the boundary, which is clearly a contradiction.	Therefore, $t_0 u_1(y_0)-u_2(y_0)=0$ at some $y_0\in \Omega$ (i.e., $t_0 u_1-u_2$ attains a local maximum at $y_0$). Then, by the maximum principle and Lemma \ref{lemma to avoid viscosity definition}, we get	
	\begin{equation}\label{eq contradiction krasnoselski ordered}
		0\geq \Delta(t_0 u_1)(y_0) -\Delta(u_2)(y_0) \geq g(|u_2\geq u_2(y_0)|)-t_0 g(|u_1\geq u_1(y_0)|).
	\end{equation}
	Thus, if $t_0 <1$, by monotonicity of $g$ we obtain a contradiction unless 
	
	$$g(|u_1\geq u_1(y_0)|)= g(|u_2\geq u_2(y_0)|)=0.$$
	
	From here, we conclude that $t_0\geq 1$ if $g(0)>0$. If not, we deduce  $\max u_2 = u_2(y_0) = t_0  u_1(y_0) = t_0 \max u_1$. \\

Notice that under $(H_1)$ we have concluded unconditionally that $t_0\geq 1$, which implies $u_2 \geq u_1$. By reversing the roles of $u_1$ and $u_2$ in the previous argument, we deduce $u_1=u_2$. Thus, we will only enforce $(H_2)$ in the rest of the proof.\\

 \medskip
    
\noindent \textit{Step 3} We show that $u_2\leq u_1$.\\
	
	If $u_1$ and $u_2$ are not comparable, there exist $s_0, t_0 \in (0,1)$ such that $t_0 u_1  \leq  u_2$ and $s_0 u_2  \leq u_1$. In this case, by the previous step,  $\max u_2 = t_0 \max u_1 = s_0 t_0 \max u_2 < \max u_2$ which is a contradiction. So, for the rest of the proof, we assume without loss of generality that $u_2 \leq u_1$ and that for some $t_0 \in (0,1)$ $u_1 t_0 \leq u_2$.\\

	\medskip
    \noindent \textit{Step 4}  Let $D_1$ and $D_2$ be the sets where $u_1$ and $u_2$ attain their maximum, respectively. We show that $D_1\subset D_2$, $|D_2\setminus D_1|=0$.\\

    Given any $x\in D_1$, $u_2(x)\geq t_0 u_1(x)= t_0 \max u_1 =\max u_2$. This shows the inclusion $D_1\subset D_2$. Also, Lemma \ref{lemma dead core} implies that $|D_i|= \alpha$ for $i=1,2$. So, $|D_2\setminus D_1| = 0$. \\

	\medskip
	\noindent \textit{Step 5} We show that $u_2\geq u_1$.\\
		
	From step 2, we have that for any $s \in (0, \max u_2)$ we have that $u_2 -t_0u_1>0$ on the compact set $\{u_2 = s\}$. Thus, by continuity of $u_1$ and $u_2$, there exists $k_s>0$ such that  $u_2 \geq t_0 u_1+ k_s$ on  $\{u_2 = s\}$. Moreover, we claim that for each $\varepsilon>0$ small enough, there exists a vector $y_\varepsilon \in B_\varepsilon(0)$ such that the translated function $ u_1^\varepsilon(x)=t_0 u_1(x+y_\varepsilon)$ satisfies	
	\begin{enumerate}
		\item  $u_2 >u_1^\varepsilon + \frac{k_s}{2}$ on $\{u_2 = s\}$,
		\item  there exists $x_\varepsilon \in \{u>s\}$ such that $\max u_2 >u_2(x_\varepsilon)$ and $u_1^\varepsilon(x_\varepsilon)= \max u_2 $.
	\end{enumerate} 
	
	Indeed, by the previous observations $dist(D_1, \{u_2\leq s\})>0$. Hence, since $|D_2\setminus D_1|=0$, we deduce that for $\varepsilon>0$ small enough and for any $y_\varepsilon \in B_\varepsilon(0)$ we have that $\left(D_1 +y_\varepsilon\right) \cap \left(\{u_2>s\}\setminus D_2\right) \neq \emptyset$, proving (2).
	On the other hand, (1) is a direct consequence of the continuity of $u_2$ and $u_1$.\\

	So, by combining properties (1) and (2) we have that $u_1^\varepsilon-u_2$ attains a positive maximum at some point $z_\varepsilon \in \{u_2>s\}$ . Therefore, a direct application of the maximum principle  (and Lemma \ref{lemma to avoid viscosity definition}) yields
	\begin{eqnarray*}
		0\geq t_0 \Delta u_1(z_\varepsilon+y_\varepsilon)- \Delta u_2(z_\varepsilon)&= &g(|u_2\geq u_2(z_\varepsilon)|)-t_0 g(|u_1\geq u_1(z_\varepsilon+
		y_\varepsilon)|)\\
		&=&g(|u_2\geq u_2(z_\varepsilon)|)-t_0 g(|t_0u_1\geq t_0u_1(z_\varepsilon+
		y_\varepsilon)|)\\
		& \geq& g(|u_1\geq u_1^\varepsilon(z_\varepsilon)|)-t_0 g(|t_0u_1\geq t_0u_1(z_\varepsilon+
		y_\varepsilon)|)\\
		& =& g(|u_2\geq t_0u_1(z_\varepsilon+
		y_\varepsilon)|)-t_0 g(|t_0u_1\geq t_0u_1(z_\varepsilon+
		y_\varepsilon)|),
	\end{eqnarray*}
    where we have used that $g$ is increasing and that $u_1^\varepsilon(z_\varepsilon)> u(z_\e)$. Thus, since $t_0<1$ and $g$ is non-decreasing, we have deduced that necessarily  $$g(|t_0u_1\geq t_0u_1(z_\varepsilon+
		y_\varepsilon)|)= g(|u_2\geq t_0u_1(z_\varepsilon+
		y_\varepsilon)|)=0,$$
        i.e., $z_\varepsilon$ is maximizer of $u_2$ and $u_1^\varepsilon$. However, in this case $u_1^\varepsilon(z_\varepsilon)=u_2(z_\varepsilon)$, which contradicts the fact that $u_1^\varepsilon-u_2$ attains a positive maximum. This implies that $u_2\geq u_1$ . \\
	
\end{proof}

 We conclude the first part of this section with the
\begin{proof}[Proof of Theorem \ref{thm 1}]
 Thanks to \cite{caffarelli2023fully} and Theorem \ref{uniqueness}, \eqref{main problem grad} has a unique solution, which is $C^{1,1}(\Omega)$ in virtue of Lemma \ref{lemma regularity Shagohlian}.
\end{proof}

\subsection{Symmetry and physical cases}

We can exploit the natural symmetries enjoyed by the right hand side of  \eqref{main problem grad} in combination with the uniqueness to characterize solutions in special domains. To make more precise this assertion, let us introduce some notation.

\begin{definition}
	Let $\Omega$ be a domain and let $G$ be its group of isometries. For a measurable function $u:\Omega\to \R$ let us define 			$u_A(x)= u(Ax)$ for $A\in G$.
\end{definition}

We have the following general symmetry result.
\begin{lemma}\label{lemma map general symmetry}
	Let $\Omega$ be a domain invariant with respect to the action of the group $G$. If $g$ is Lipschitz and non-decreasing, then the solution $u$ to \eqref{main problem grad} is invariant under the action of $G$.
\end{lemma}

\begin{proof}
	Let $A\in G$. The proof follows from noticing that $u_A$ also solves \eqref{main problem grad}.  Notice that the measure of the superlevel sets of $u$ is invariant with respect to the action of $G$, i.e., for any $x\in \Omega$ we have
\begin{equation}\label{eq: invariance rotations}
	|y\in \Omega : u(Ay)\geq u(x)|=|y\in \Omega : u(y)\geq u(x)|.
\end{equation} 

Since $\Delta$ commutes with isometries, we have that
\begin{equation*}
	\Delta u_A(x) = \Delta u(Ax) = g(|y\in \Omega : u(y)\geq u(Ax)|)=g(|y\in \Omega : u(Ay)\geq u(Ax)|).
\end{equation*} 
Thus, Theorem \ref{uniqueness} implies that $u_A=u$.
\end{proof}

\begin{remark}
  The previous proof works (identically) for more general elliptic operators as long as the operator verifies the hypotheses discussed in the preamble to the proof of Theorem \ref{uniqueness} and commutes with the group of isometries of the given domain $\Omega$. For instance, if $\Omega$ is radially symmetric (a ball or an annulus), $G = O(n)$. In this case, a good source of $O(n)$-invariant operators are those depending only on the eigenvalues of the Hessian of functions, since $\lambda(D^2 u_A)(x)=\lambda(D^2 u)(Ax)$, for any eigenvalue $\lambda(D^2 u)$ of $D^2 u$. Examples of this kind of operators are the Monge-Ampere operator, and maximal Pucci operators (see \cite{caffacabre}). Another family of operators that satisfy this property are the $p$-Laplacian operators or, even more general, any of the form $\text{div}(a(|\nabla u|) \nabla u)$.  
\end{remark}

In the next lemma, we exploit the radial symmetry to provide examples showing the optimal regularity for solutions of \eqref{main problem grad} in the presence of dead cores. In particular, we show that the solutions detach in a $C^{2,1}$ way from their maximum, suggesting the optimal regularity for this problem. Let us remark that, even in the absence of dead cores, solutions are not necessarily smooth at their maxima (see \cite[Example 2.4]{caffarelli2023fully}).

\begin{proposition}\label{corollary radial symmetry lip}
	Let $\Omega$ be a ball and let $g$ be a smooth non-decreasing function satisfying $(H_2)$ with $\alpha>0$. Let $u$ be the unique solution of \eqref{main problem grad}, then $ \nabla u \neq 0$ in $\Omega \setminus D_u$ and
    	\begin{equation}\label{eq cubic bound}
	 \max u-u\leq  C\text{dist}(x, D_u)^3, \quad \mbox{ in $\Omega$}.
		\end{equation}
    Furthermore, if $g'(t)\geq \frac{1}{C}$ for $t \in [\alpha, |\Omega|]$, we have that 
    \begin{equation}\label{eq cubic detachment}
	 \max u-u\geq  \frac{1}{C}\text{dist}(x, D_u)^3, \quad \mbox{ in $\Omega$},
		\end{equation}
        and that \eqref{eq AP comparison} holds.
\end{proposition}
\begin{proof}
 Without loss of generality, we assume $\Omega = B_R$. By Lemma \ref{lemma map general symmetry}, $u$ is radially symmetric. In particular, $D_u$ is connected; otherwise, $u$ will forcefully attain a local minimum contradicting the superharmonicity of $u$. Hence, $\{u=u_{max}\}= \overline{B_{r}(0)}$ and, thanks to Lemma \ref{lemma dead core}, $r(\alpha)=\left({\frac{\alpha}{\omega_n}}\right)^\frac{1}{n}$, where $\omega_n$ is the Lebesgue measure of the unit ball.  Furthermore, since the radial profile of $u$ is even, increasing in $(-R,-r(\alpha))$, and decreasing in $(r(\alpha),R)$, we deduce that
	\begin{equation}\label{eq radial nonlin}
		f(x)=	g(|u\geq u(x)|)=\begin{cases}
			g\left( \omega_n |x|^n \right), &|x|\geq r(\alpha)\\
			g(\alpha)=0
			, & |x|\leq r(\alpha).
		\end{cases}
	\end{equation}
   In particular, $f$ is Lipschitz continuous in $\Omega$ and smooth in $\Omega \setminus N_u$. Therefore, by classical elliptic regularity, $u \in C^{2,\gamma}(\Omega)$ for any $\gamma \in (0,1)$ and $u \in C^\infty_{loc}(\Omega\setminus N_u)$. Now, set $u(x)=\zeta(|x|)$. By the Lipschitz regularity of $g$, we have that 
   \begin{equation}\label{eq lipchitz f}
       f(|x|)=g(\omega |x|^n)-g(\alpha)\leq C(|x|-r(\alpha))_+.
   \end{equation}

   Additionally, we have that $\zeta$ satisfies	
	\begin{equation*}
		-\big( r^{n-1}\zeta'\big)' = r^{n-1}f(r)
	\end{equation*}	
	on $[r(\alpha), R]$, with $\zeta''(r(\alpha))=u'(r(\alpha))=0$ and $\zeta(r(\alpha))=\max u$. Integrating between $r(\alpha)$ and $r$ we find
	\begin{equation}\label{eq main zeta}
	    -\zeta'(r)= \frac{1}{r^{n-1}}\int_{r(\alpha)}^{r} s^{n-1} f(s)ds.
	\end{equation} 
	By combining \eqref{eq lipchitz f} with  \eqref{eq main zeta}, and integrating one more time between $r(\alpha)$ and $s$ with $s<R$, we readily deduce that  
    $$\max u - \zeta(s)= \int_{r(\alpha)}^s\frac{1}{r^{n-1}}\int_{r(\alpha)}^{r} s^{n-1} f(s)dsdr\leq C(s-r(\alpha))^3,$$
    which yields \eqref{eq cubic detachment}. By combining \eqref{eq cubic detachment} with the smoothness of $u$ in $\Omega\setminus N_u$ and its radial symmetry, we deduce that $u \in C^{2,1}(\Omega)$.\\
    
    Furthermore, under the assumption $g'(t)\geq \frac{1}{C}$ for $t \in [\alpha, |\Omega|]$, we have that  
    $$f(|x|)=g(\omega |x|^n)-g(\alpha)\geq \frac{1}{C}(|x|-r(\alpha))_+,$$
   which combined with \eqref{eq main zeta} yields, by an analogous integration argument,  
   $$\max u - \zeta(s) \geq \frac{1}{C}(s-r(\alpha))^3.$$	
\end{proof}

\section{Free boundary}\label{sec: fb}   
\subsection{Free boundary analysis preliminaries}
Our aim now is to study free boundary regularity for solutions of \eqref{main problem grad} together with their qualitative behavior as they approach the free boundary, which we will denote by $FB(u) = \partial \{u < \max u\}$. For this reason, we will enforce from now on the hypothesis $(H_2)$ with $\alpha>0$. Additionally, in this section, we will choose to work with the function $v= \max u - u$, so that now $\{u = \max u\} = \{v =0\}$ and $FB(u) = \partial \{v > 0\} := FB(v)$. This transformation is also convenient, since it lets us rewrite \eqref{main problem grad} as
\begin{equation}\label{eq reversed problem}
	\Delta v = f(v), \quad \mbox{in $\Omega$},
\end{equation}
where $f(t)= g(|v\leq t|)$ is an increasing and absolutely continuous function (the latter in virtue of Lemma \ref{lemma abs cont}). In particular, this allows us to freeze the non-local dependence of $v$ and treat it as a semilinear free boundary problem. Our main strategy will be to show that solutions to \eqref{eq reversed problem} behave near their free boundary as solutions to the so-called Alt-Phillips problem (\cite{alt1986free}). Given $\gamma \in (1,2)$, a non-negative continuous function $v_0$ vanishing at $D$ satisfies the Alt-Phillips equation in $\Omega$ if 
\begin{equation}\label{eq altphil}
	\begin{cases}
		\Delta v_0 = v_0^{\gamma-1}, \quad \mbox{in $\Omega$},\\
		|\nabla v_0| =0, \quad \mbox{on $D\cap \Omega$}.
	\end{cases}   
\end{equation}
It is known (\cite{alt1986free}, \cite{bonorino2001regularity}) that for any $x_0 \in \partial D$, and for $r$ small enough such that $B_r(x_0) \subset \Omega$,
\begin{equation}\label{eq eqdistancealt}
	\frac{1}{C} r^\frac{2}{2-\gamma} \leq \sup_{B_r(x_0)} w \leq C r^\frac{2}{2-\gamma}.
\end{equation}

In our case, as may be seen in the examples discussed in Section 3, the optimal regularity for $v$ near its free boundary is $C^{2,1}$, corresponding to $\gamma = \frac{4}{3}$ in the Alt-Phillips equation. Heuristically, this can be thought of as the third derivatives of $v$ having a jump across its free boundary, whereas $v, \nabla v$, and $\Delta v$ continuously transition to zero across $FB(v)$. \\

The first result confirming this intuition shows that solutions to \eqref{eq reversed problem} are subsolutions to an Alt-Phillips type equation with $\gamma = \frac{4}{3}$.

\begin{lemma}\label{lemma nondegeneracy}
	Let us assume that $g$ satisfies $g'(t) \geq \frac{1}{C}$ for a.e. $t\in [\alpha, |\Omega|]$. If $v$ solves \eqref{eq reversed problem}, then
	\begin{equation}\label{eq non-degeneracy}
		g(|v\leq t|)\geq \frac{1}{C}t^\frac{1}{3},
	\end{equation}
	in $[0, \max v]$. 
\end{lemma}
\begin{proof}
	Set $h(t)=|v\leq t|$. Thanks to Lemma \ref{lemma abs cont}, we have that $h(t)$ is absolutely continuous and for almost every $t$ in $[0, u_{\max}]$ we also have that
	\begin{equation}\label{eq derivativemu}
		h'(t) = \int_{ \{v=t\}} \frac{1}{|\nabla v|} d\mathcal{H}^{n-1}.
	\end{equation}
	Additionally, since $v$ is Lipschitz, we can take values of $t$ such that $\{v<t\}$ is a set of finite perimeter. Therefore, by the Cauchy Schwarz inequality,
	\begin{equation}\label{eq perimeterbound}
		\text{Per}(v<t) = \int_{\{ v=t\}} \frac{|\nabla v|^\frac{1}{2}}{|\nabla v|^\frac{1}{2}}d\mathcal{H}^{n-1}\leq \Big(\int_{ \{ v=t\}} \frac{1}{|\nabla v|} d\mathcal{H}^{n-1}\Big)^\frac{1}{2} \Big(\int_{\{ v=t\}} |\nabla v| d\mathcal{H}^{n-1}\Big)^\frac{1}{2}.
	\end{equation}
	
	On the other hand, given $t>0$, we can test in \eqref{eq reversed problem} with $(t-v)_+$ to deduce
	\begin{equation}\label{eq testinlevelsets}
		\int_0^t \int_{\{ v=s\}} |\nabla v| d\mathcal{H}^{n-1}ds=\int_{\{ v\leq t\}} |\nabla v|^2 d\mathcal{H}^{n-1}= \int_{\{v\leq t\}} g(|v\leq v(x)|)(t-v(x))_+.
	\end{equation}
	
	Thus, taking derivatives on both sides of \eqref{eq testinlevelsets}, we deduce that for a.e. $t$ in $[0, u_{\max}]$,
	\begin{equation}\label{eq derivativelevel}
		\int_{\{ v=t\}} |\nabla v| d\mathcal{H}^{n-1} = \int_{\{v\leq t\}} g(|v\leq v(x)|)dx.
	\end{equation}
	
	Hence, the combination of \eqref{eq derivativemu}, \eqref{eq perimeterbound}, and \eqref{eq derivativelevel} provides the inequality
	\begin{equation}\label{eq differentialineq0}
		h'(t) \geq \text{Per}(v<t)^2 \Big( \int_{\{v\leq t\}} g(|v\leq v(x)|)\Big)^{-1}.
	\end{equation}
	
	Since $g$ is Lipschtiz and $g(|v=0|)=0$, we deduce that
	\begin{equation}\label{eq integralbound}
		\int_{\{v\leq t\}} g(|v\leq v(x)|)\leq  C \int_{\{0<v\leq t\}} |0<v< v(x)|\leq C |0<v< t|^2.
	\end{equation}
	Further, in virtue of the isoperimetric inequality and the fact that $|v=0|=\alpha>0$, we deduce
	$$\text{Per}(v<t)\geq \frac{1}{C}|v<t|^\frac{n-1}{n}\geq \frac{1}{C} (\alpha)^\frac{n-1}{n}.$$
	
	Combining this latter remark with \eqref{eq integralbound} and \eqref{eq differentialineq0} we deduce the differential inequality
	\begin{equation}\label{eq differentialineq0}
		h'(t) \geq \frac{1}{C} h(t)^{-2}.
	\end{equation}

    Let us set $f(t)= g(|v\leq t|)$. Since, by hypothesis, $g$ is bilipschitz, we can rewrite \eqref{eq differentialineq0} as
    \begin{equation}\label{eq differentialineq1}
		f'(t) \geq \frac{1}{C} f(t)^{-2}.
	\end{equation}
	
	Using the absolute continuity of $f$ and integrating \eqref{eq differentialineq1} between $\e$ and $t$, we deduce
	\begin{equation}\label{eq differentialineq0}
		f(t)^3-f(\e)^3 \geq \frac{1}{C} (t-\e),
	\end{equation}
	which yields the result after taking $\e\to 0^+$ and using the continuity of $f$ at 0. 
\end{proof}

The next step is to show that $v$ is a supersolution to an Alt-Phillips equation with exponent $\gamma  = \frac{7}{6}$. This is the first step in an iterative process that will culminate in showing that $v$ is a supersolution of an Alt-Phillips equation with exponent $\frac{1}{5}$. With this goal in mind, let us recall some gradient estimates proved in \cite{phillips1983hausdoff} for solutions to the Alt-Phillips equation. We claim that the same proof also holds for supersolutions and, as the complete proof is dispersed across various papers, we present it here for the sake of clarity and completeness.

\begin{lemma}\label{lemma grad bound 0}
Let $\gamma \in [1,2)$ and let $v \in C^{1,1}(\Omega)$ be a non-negative function. Let $D =\{v =0\}$ and assume that $v$ satisfies
\begin{equation}\label{eq altphil supersol}
	\begin{cases}
		\Delta v \leq C m_\gamma(v), \quad \mbox{in $\Omega$},\\
		|\nabla v| =0, \quad \mbox{on $D\cap \Omega$},
	\end{cases}   
\end{equation}
with $m_\gamma(t) = t_+^{\gamma-1}$ if $\gamma >1$ and $\chi_{\{t>0\}}$ if $\gamma=1$. Then, we have that
\begin{equation}\label{eq 0 gradient estimate}
    |\nabla v(x)| \leq C(v(x))^{\frac{\gamma}{2}}.
\end{equation}
\end{lemma}
\begin{proof}
 We prove the case $\gamma>1$ and we refer the reader to \cite{phillips1983hausdoff} for the remaining case $(\gamma=1)$.\\ 

\noindent{\it Step One:} We show the following form of non-degeneracy: assume $0\in \Omega \setminus D$ and let $r>0$ be such that $B_{r}(0) \subset \subset \Omega$, then there exist universal constants $c_0, \tau > 0$ such that if $\fint_{\partial B_r} v \geq c_0r$, then $v(x) \geq \tau(\int_{\partial B_r} v)$ for $x \in B_{\frac{r}{4}}(0)$.\\

This first step follows \cite[Lemma 1.1 ]{phillips1983hausdoff}. Consider
\begin{equation*}
G(r) = 
\begin{cases}
    -\log r \hspace{8mm} \text{ if } n = 2 \\
    r^{2 - n} - 1 \hspace{5mm} \text{ if } n \geq 3,
\end{cases}
\end{equation*}
which is , up to a constant, the fundamental solution of $\Delta$. Consequently, integrating by parts and using the boundedness of $f$, we have that for some universal constants $c, \bar{c}$,
\begin{equation*}
\fint_{\partial B_1} v - v(0) = c\int_{B_1} G\Delta v \leq c\Vert f(v)\Vert_{L^\infty(\Omega)}\int_{B_1} G \leq \bar{c}.
\end{equation*}
Then, we deduce that, if 
\begin{equation}\label{harnacksito 1}
\fint_{\partial B_1} v \geq 2\bar{c} := c_0, 
\end{equation}
then
\begin{equation}\label{harnacksito 2}
v(0) \geq \frac{1}{2}\fint_{\partial B_1} v.
\end{equation}
Set $x \in B_{\frac{1}{4}}$. Since $v$ is subharmonic, we get
\begin{equation}\label{eq subharm bound}
v(0) \leq c(n)\fint_{\partial B_{\frac{1}{2}}(x)} v.
\end{equation}
Then, if we assume that
\begin{equation*}
\fint_{\partial B_1(0)} v \geq c(n)c_0,
\end{equation*}
it follows that
\begin{equation*}
\fint_{\partial B_{\frac
{1}{2}(x)}} v \geq c_0.
\end{equation*}
Thus, by combining \eqref{harnacksito 1} and \eqref{harnacksito 2}, we obtain
\begin{equation*}
v(x) \geq \frac{1}{2}\int_{\partial B_{\frac{1}{2}}(x)} v.
\end{equation*}
Putting everything together, we get that for some universal constant $\tau >0$,
\begin{equation*}
v(x) \geq \tau \fint_{\partial B_{1}(0)} v,
\end{equation*}
as long as 
\begin{equation*}
\fint_{\partial B_1(0)} v \geq c_0.
\end{equation*}

\medskip

\noindent{\it Step two:} We show that for any $x_0 \in \Omega$ and $r > 0$ such that $B_{4r}(x_0) \subset \subset \Omega$, we have the Harnack inequality

\begin{equation}\label{eq harnack ap}
\sup_{x \in B_{r}(x_0)} w \leq C[\inf_{x \in B_r(x_0)} w + r^\frac{2}{2-\gamma}].
\end{equation}

Without loss of generality, we may assume that $x_0 = 0$. Then, if 
\begin{equation*}
\fint_{\partial B_{4r}(0)} v \leq c_0r^{\frac{2}{2 - \gamma}},    
\end{equation*}

we have by subharmonicity that 

\begin{equation*}
\sup_{x \in B_r(x_0)} v \leq Cr^{\frac{2}{2 - \gamma}}.   
\end{equation*}

Otherwise, we have that 

\begin{equation*}
v(x) \geq \tau \fint_{\partial B_r} v,    
\end{equation*}

and so 

\begin{equation*}
\sup_{x \in B_r} v(x) \leq C\inf_{x \in B_r} v(x).    
\end{equation*}

Combining the two estimates yields \eqref{eq harnack ap}.\\

\medskip

\noindent{\it Step three:} We show \eqref{eq 0 gradient estimate}. This part follows \cite[Lemma 1.12]{alt1986free}.\\

Since $w(0)>0$, we can take $r$ such that $r^\frac{2}{2-\gamma} = w(0)$. Then, \eqref{eq harnack ap} implies
\begin{equation}\label{harnacksito 3}
\sup_{x \in B_r} w \leq C w(0).
\end{equation}

By gradient estimates for the Laplacian, we have that 
\begin{equation*}
\sup_{x \in B_\frac{r}{2}} |\nabla w| \leq C[r\sup_{x \in B_r}m_\gamma(w) + \frac{1}{r}\sup_{x \in B_r} w]. 
\end{equation*}

Then, by applying \eqref{harnacksito 3}, and the choice of $r$, we deduce
\begin{equation*}
\sup_{x \in B_\frac{r}{2}} |\nabla v(x)| \leq C[w(0)^\frac{2}{2-\gamma}w(0)^{\gamma-1}+w(0)^\frac{-2}{2-\gamma}w(0)]=Cw(0)^\frac{\gamma}{2},
\end{equation*}
finishing the proof.
\end{proof}

Additionally, assuming that $v$ is a supersolution of an Alt-Phillips equation, we provide a bound on the second derivatives of $v$ that matches the known bound for the Alt-Phillips equation (again, see \cite{phillips1983hausdoff} and \cite{alt1986free}). The argument is essentially a modification of an argument of Shahgholian \cite{shahgholian2003c1}.

\begin{lemma}\label{shahgholian bound}
Let $v$ be a solution to \eqref{eq reversed problem} and furthermore assume that $v$ satisfies \eqref{eq altphil supersol} with $\gamma > 1$. Then, we have that 

\begin{equation}\label{eq shahgholian}
|D^{i,j}v(x)| \leq C|v(x)|^{\gamma - 1}    
\end{equation}

for a.e. $x \in \Omega$. 
\end{lemma}
\begin{proof}
Let us recall the Alt-Caffarelli-Friedman monotonicity functional centered at the origin, $h(r, u)$, which is valid for a subharmonic function $f$
\begin{equation}
h(r,f) = \frac{1}{r^2}\int_{B_r} \frac{|\nabla f(x)|^2}{|x|^{n-2}}. 
\end{equation}
It is known that if $f_1$ and $f_2$ are Lipschitz subharmonic functions in the unit ball such that $f_1(0) = f_2(0) = 0$ and $f_1f_2 = 0$, then the functional
\begin{equation}
\phi(r,f_1,f_2) = \frac{1}{r^4}\int_{B_r} \frac{|\nabla f_1(x)|^2}{|x|^{n-2}}\int_{B_r} \frac{|\nabla f_2(x)|^2}{|x|^{n-2}}
\end{equation}
is finite and nondecreasing for $0 < r < 1$. Now, if we derive the equation \eqref{eq reversed problem} in the direction $e$, we obtain
\begin{equation}
\Delta v_{e} = f'(v)v_{e}.
\end{equation}
If we assume that $v_{e}$ obtains a negative minimum at some $y \in \Omega$, then $\Delta v_e < 0$ at $y$, a contradiction. Hence, we conclude that the derivatives of $v$ are subharmonic functions. Now, suppose that $0 \in \{v > 0\}$ is such that $v$ is twice differentiable at such point. Then, let $\tilde{v}(x) = v_{e}(x)$, where $e$ is any vector orthogonal to $\nabla v(0)$. Let $\tilde{v}^{+}$ and $\tilde{v}^{-}$ denote the positive and negative parts of $v$, respectively. Then, it follows that 
\begin{eqnarray*}
\lim_{r \rightarrow 0} \phi(r, \tilde{v}^{+}, \tilde{v}^{-}) &\geq& \int_{B_1(x_0)} \lim_{r \rightarrow 0} \frac{|\nabla \tilde{v}^{+}(yr)|^2}{|y|^{n-2}}\int_{B_1} \lim_{r \rightarrow 0} \frac{|\nabla \tilde{v}^{-}(yr)|^2}{|y|^{n-2}} \\
&\geq& \int_{K_1} \lim_{r \rightarrow 0} \frac{|\nabla \tilde{v}^{+}(yr)|^2}{|y|^{n-2}}\int_{K_2} \lim_{r \rightarrow 0} \frac{|\nabla \tilde{v}^{-}(yr)|^2}{|y|^{n-2}},    
\end{eqnarray*}

where $K_1$ and $K_2$ are cones centered at the origin such that $\tilde{v}^{+} = \tilde{v}$ in $K_1$ and $\tilde{v}^{-} = \tilde{v}$ in $K_2$. Letting $r \rightarrow 0$ then yields that 
\begin{equation}\label{eq Shahgholian bound 1}
|\nabla \tilde{v}(0)|^4 \leq C\lim_{r \rightarrow 0}\phi(r, \tilde{v}^{+}, \tilde{v}^{-}).
\end{equation}

Now, if $|\nabla v(0)| = 0$, then we may take $e$ arbitrary and conclude that
\begin{equation}\label{eq Shahgholian bound 2}
|D^{i,j}v(0)|^{4} \leq C\lim_{r \rightarrow 0}\phi(r, \tilde{v}^{+}, \tilde{v}^{-}),
\end{equation}

for all $i,j \leq n$. Otherwise, letting $\{e_i\}$ be the standard system of coordinates, we may assume that $\nabla v(0)$ is parallel to $e_1$. Then, \eqref{eq Shahgholian bound 1} implies that
\begin{equation}\label{eq shahgholian ij}
|\partial_{i,j}v(0)|^{4} \leq C\lim_{r \rightarrow 0}\phi(r, \tilde{v}^{+}, \tilde{v}^{-}) \qquad \text{ for } i = 2,...,n, j = 1,...,n.
\end{equation}
But then we use \eqref{eq reversed problem} along with the bound \eqref{eq altphil} supersolution to get that

\begin{equation}\label{eq shahgholian 11}
|\partial_{1,1}v(0)|^{4} \leq C\sum_{i = 2} ^ {n} |D^{i,i}v(0)|^{4} + Cv^{4(\gamma - 1)} \leq C(\lim_{r \rightarrow 0}\phi(r, \tilde{v}^{+}, \tilde{v}^{-}) + v^{4(\gamma - 1)}).
\end{equation}
We then notice by Holder's inequality (with $p= \frac{n-1}{n-2}$ and $q= n-1$), \eqref{eq altphil supersol}, the Calderon Zygmund inequality from elliptic theory, and the Harnack inequality \eqref{eq harnack ap}
\begin{eqnarray*}
\phi(r, v^{+}, v^{-}) &\leq& C\Big( \frac{1}{r^{n/q-1}} \Vert D^2v\Vert_{L^q(B_r)}\Big)^4 \\ &\leq& C \Big(\frac{1}{r^{n/q-1}} (\Vert \Delta v\Vert_{L^q(B_{2r})} + r^{-2} \Vert v\Vert_{L^q(B_{2r})})\Big)^4 \\
&\leq& C\Big(||v||_{L^{\infty}(B_{2r})}^{\gamma - 1} + r^{-2}||v||_{L^{\infty}(B_{2r})}\Big)^4 \\
&\leq& C(v(0)^{\gamma - 1} + r^{-2}v(0))^{4}.
\end{eqnarray*}
As in Lemma \ref{lemma grad bound 0}, we choose $r^{\frac{2}{2-\gamma}} = 0$. Then, we use the Harnack inequality \eqref{eq harnack ap} to get that $r^{-2}||v||_{L^{\infty}(B_{2r})} \leq Cr^{-2}v(0) = Cv(0)^{\gamma - 1}$. Hence, we conclude that
\begin{equation*}
\phi(r, v^{+}, v^{-}) \leq Cv(0)^{4(\gamma - 1)},
\end{equation*}
which, along with \eqref{eq shahgholian ij} and \eqref{eq shahgholian 11}, proves the result. 
\end{proof}

The next lemma is similar in spirit to Lemma 4 in \cite{caffarelli1998obstacle} and will allow us to prove the sub-optimal supersolution bound in conjunction with the first gradient bound \eqref{lemma grad bound 0}.

\begin{lemma}\label{lemma caffarelli gradient bound}
Let $x_0 \in FB$ and let $r_0 = \text{dist}(x_0, \partial \Omega)$.  Then, for any $r\in (0, r_0)$, $h > 0$ and $i=1,\cdots, n$ 
\begin{equation}\label{eq dirichlet bound}
\int_{\{0 \leq |v_i| \leq h\} \cap B_r(x_0)} |\nabla v_i|^2 dx\leq Chr^{n-1}.
\end{equation}
\end{lemma}
\begin{proof}
Let us assume without loss of generality $x_0=0$. Also, given a function $f$, let us fix the notations 
$$f_h:= \min\{f, h\}_+, \mbox{ and } T_{e,t}f(x) := \frac{f(x + te) - f(x)}{t},$$
where $t>0$ and  $e\in \mathbb{S}^{n-1}$. Also, let us set $G(s) := g(|0 \leq v \leq s|)$, which, by our hypothesis on $g$, is a non-decreasing function.\\

Integrating by parts
\begin{equation}\label{eq int parts dq}
\int_{B_r} \Delta({T_{e,t}v})(T_{e,t}v)_h dx = \int_{\partial B_r} (T_{e,t}v)_h \partial_{\nu}T_{e,t}v - \int_{B_r}|\nabla (T_{e,t}v)_h|^2 dx.
\end{equation}
Additionally, by \eqref{eq reversed problem}
\begin{equation*}
\Delta({T_{e,t}v})(x) = G(v(x + te)) - G(v(x))
\end{equation*} 
in $B_r$ for $t>0$ small enough such that $x+te\in \Omega$ for any $x\in B_r$. Then, by the monotonicity of $G$, we have that
\begin{equation}\label{eq monotonicity RHS}
     \Delta({T_{e,t}v})(x)(T_{e,t}v)_h \geq 0.
\end{equation}

Piecing out \eqref{eq int parts dq} and \eqref{eq monotonicity RHS} yields
\begin{equation*}
 \int_{B_r}|\nabla (T_{e,t}v)_h|^2 dx\leq h\int_{\partial B_r}|\nabla T_{e,t}v|  d\mathcal{H}^{n-1}.
\end{equation*}

Then, thanks to the $C^{1,1}$ regularity of solutions proved in Lemma \ref{lemma regularity Shagohlian}, we let $t\to 0^+$ and conclude
\begin{equation*}
\int_{B_r \cap \{0 \leq v_e \leq h\}} |\nabla v_e|^2dx \leq h\int_{\partial B_r }  |\nabla v_e|d\mathcal{H}^{n-1} \leq Chr^{n - 1},
\end{equation*}
where $v_e = \partial_{e} v$. So, the result follows from adding up the previous inequality with $e= \pm e_i$ for any $i=1,\cdots, n$. 
\end{proof}

Now, we are in a position to show that $v$ is a supersolution of an Alt-Phillips equation with $\gamma - 1 = \frac{7}{6}$.

\begin{lemma}\label{lemma first OM estimate}
Let $x_0 \in FB(v)$ and let $r_0 = \text{dist}(x_0, \partial \Omega)$. If $v$ solves \eqref{eq reversed problem}, then for any $r\in (0, r_0)$,
\begin{equation}\label{eq lc rep}
    |\{0 < v < h^2\} \cap B_r(x_0)|^{3}\leq C \sum_{k = 1} ^ {n} \int_{\{0 < |v_{k}| < Ch\} \cap B_r} |\nabla v_{k}|^2.
\end{equation}

In particular, 
\begin{equation}\label{eq initial bound dist}
      \Delta v\leq C v^\frac{1}{6}.
\end{equation}
\end{lemma}
\begin{proof}
Let us take $x_0=0$ and consider the function $f(t) = |\{0 < v < t\} \cap B_r|$. Since $f$ is absolutely continuous in virtue of Lemma \ref{lemma abs cont}, we apply the layer-cake representation to $f^2$ and deduce
\begin{equation}\label{eq rep}
    \int_{\{0 < v < h\} \cap B_r} f(v(x))^2 dx = \int_{0} ^ {h} 2f(s) f'(s)|s < v < h \cap B_r| ds.
\end{equation}

Further, computing the right hand side in \eqref{eq lc rep} yields
\begin{eqnarray}\notag
\int_{0} ^ {h} 2f'(s)f(s) |\{s < v < h\} \cap B_r|ds&=& \int_{0} ^ {h} 2f'(s)f(s)(f(h) - f(s)) ds\\\notag
&=& f(h)^3- 2\int_{0} ^ {h} f'(s)f(s)^2 ds\\\label{eq layercake}
&=& \frac{1}{3}f(h)^3 .  
\end{eqnarray}

Hence, \eqref{eq layercake} and \eqref{eq rep} combined with \eqref{eq reversed problem} imply
\begin{equation}\label{eq rep}
   |\{0 < v < h\} \cap B_r|^3 \leq C\int_{\{0 < v < h\} \cap B_r} (\Delta v)^2 dx .
\end{equation}

On the other hand, from Lemma \ref{lemma grad bound 0} we have that  $|\nabla v| \leq Cv^{\frac{1}{2}}$, which in turn implies that 
\begin{equation}\label{eq containment}
    S_h \subset \{|\nabla v| < Ch^\frac{1}{2}\} \subset \bigcap_{i=1}^n \{ |v_i| < Ch^\frac{1}{2}\}.
\end{equation}

Altogether, \eqref{eq rep} and \eqref{eq containment} imply that
\begin{eqnarray*}
   |\{0 < v < h^2\} \cap B_r|^3&\leq& C\int_{\{0 < v < h^{2}\} \cap B_r} (\Delta v)^2 dv \\
   &\leq& C\int_{\{0 < v < h^{2}\} \cap B_r} |D^2v|^2dx\\
   &\leq& C\sum_{k = 1} ^ {n} \int_{\{  |v_{e_i}| < Ch\} \cap B_r} |\nabla v_{e_k}|^2.
\end{eqnarray*}
From this estimate we deduce \eqref{eq lc rep}.\\

Let $h_0 = \frac{\max v}{2}$. By the Lipschitz continuity of $v$, we can cover $\overline{S_{h_0}}=\{v \leq  h_0\}$ with balls $B_r(x_i)$ with $r$ depending only on $h_0$ and such that $B_r(x) \subset \Omega$. Then, by compactness  $$\overline{S_{h_0}} \subset \bigcup_{k=1}^L B_{r}(x_k).$$

So, given any $h\leq h_0$, we can combine \eqref{eq lc rep} and  \eqref{eq dirichlet bound} to deduce
\begin{equation*}
    |0 < v \leq h| \leq \sum_{k=1}^L|\{0 < v < h\} \cap B_r(x_k)| \leq C h^\frac{1}{6}.
\end{equation*}
Upon increasing the constant $C$, we can remove the constraint in $h$, showing \eqref{eq initial bound dist}.
\end{proof}

The previous estimate can be iterated to improve the supersolution bound.

\begin{theorem}\label{theorem suboptimal supersol}
Let $v$ be a solution to \eqref{main problem grad} that satisfies \eqref{eq altphil supersol} for some $\gamma \geq 1$. Then, we have that 
\begin{equation}\label{eq improved altphillips supersol}
\Delta v \leq Cv^{\frac{\gamma}{6}}.
\end{equation}
In particular, we have that
\begin{equation}\label{eq suboptimal supersol}
\Delta v \leq Cv^{\frac{1}{5}}.
\end{equation}
\end{theorem}
\begin{proof}
Reasoning as in Lemma \ref{lemma first OM estimate} and using Lemma \ref{lemma grad bound 0}, we get that
\begin{eqnarray*}
   |\{0 < v < h^{\frac{2}{\gamma}}\} \cap B_r|^3&\leq& C\int_{\{0 < v < h^{\frac{2}{\gamma}}\} \cap B_r} (\Delta v)^2 dv \\
   &\leq& C\int_{\{0 < v < h^{\frac{2}{\gamma}}\} \cap B_r} |D^2v|^2dx\\
   &\leq& C\sum_{k = 1} ^ {n} \int_{\{  |v_{e_i}| < Ch\} \cap B_r} |\nabla v_{e_k}|^2.
\end{eqnarray*}
This then implies \eqref{eq improved altphillips supersol}. 

Iterating this argument, we conclude that $\Delta v \leq Cv^{\gamma_k - 1}$, where $\gamma_k = 1 + \frac{\gamma_{k-1}}{6}$ and $\gamma_1 = \frac{1}{6}$. Since the fixed point of this sequence is $\frac{1}{5}$, we deduce \eqref{eq suboptimal supersol}.
\end{proof}

The fact that $v$ is a supersolution of an Alt-Phillips equation allows us to deduce integrability for $v^{-1}$, as the next lemma shows.
\begin{lemma}\label{lemma integrability v^{-1}}
Assume that $v$ is a solution to \eqref{eq reversed problem} and that it is a supersolution of an Alt-Phillips equation with some exponent $\gamma - 1$. Then, $v^{-1} \in L^{p}(\{v > 0\})$ for all $p < \gamma - 1$.   
\end{lemma}
\begin{proof}
Using the coarea formula, we compute that
\begin{equation}
\int_{\{v > 0\}} v^{-p} \leq C\int_{0} ^ {\max v} \frac{f'(s)}{s^{p}} ds \leq C\int_{0} ^ {\max v} \frac{f'(s)}{f(s)^{\frac{p}{\gamma - 1}}} ds.      
\end{equation}
The above quantity is finite whenever $\frac{p}{\gamma - 1} < 1$, or when $p < \gamma - 1$.
\end{proof}

\medskip

We are in position now to show that the finite $\H$-measure of the free boundary and the optimal upper bound for $f(t)$ are somewhat equivalent.

\begin{proof}[Proof of Theorem \ref{thm equivalence}]
First, assume the optimal supersolution bound. Let $c(t) = \H(v = t)$. Suppose, by way of contradiction, that $\liminf_{t \rightarrow 0} c(t) = \infty$. From \eqref{eq differentialineq0}, we have that
\begin{equation*}
f^{3}(t) \geq C\int_{0} ^ {t} c(s)^2 ds.
\end{equation*}

Additionally, we repeat the proof of Lemma \ref{lemma caffarelli gradient bound} with $\{0 \leq |v_i| \leq h\} \cap B_{r}(x_0)$ and the proof of Lemma \ref{lemma first OM estimate} applied now to the entire sublevel set $\{0 < v < h\}$ instead of $\{0 < v < h \cap B_r\}$. This yields the estimate

\begin{equation}\label{eq nonlocal caffa}
f^{3}(t) \leq C\int_{v = t} |\nabla v||D^2v| d\H.   
\end{equation}
By applying \eqref{eq 0 gradient estimate} and \eqref{shahgholian bound}, we have that
\begin{equation*}
f^{3}(t) \leq Ctc(t).
\end{equation*}
Then, we apply Jensen's inequality to get that
\begin{equation}\label{eq average level set}
\fint_{0} ^ {t} c(s) ds \leq C\sqrt{c(t)}
\end{equation}
for all $t$ such that $t$ is less than the maximum value of $v$. Then, given any large constant $N$, there must exist some $t_0$ such that $c(t) \geq c(t_0)$ for $t < t_0$ and $c(t_0) \geq N$. But then,  
\begin{equation*}
\fint_{0} ^ {t_0} c(s) ds \geq c(t_0),
\end{equation*}
which implies by \eqref{eq average level set} that  
\begin{equation*}
c(t_0) \leq C.
\end{equation*}
Take $N \geq C$ to obtain a contradiction.

Conversely, suppose that there exists a sequence $\{a_k\}$ such that $\H(\{v = a_k\}) \leq C$. Choose $t > 0$, and let $k$ be such that $a_{k + 1} \leq t \leq a_k$. Let $\gamma \leq \frac{4}{3}$ be such that $f(t) \leq Ct^{\gamma - 1}$. Then, notice that by \eqref{eq nonlocal caffa}, \eqref{eq 0 gradient estimate}, and \eqref{eq shahgholian} 
\begin{equation*}
f^{3}(t) \leq f^{3}(a_k) \leq \int_{v = a_k} |\nabla v||D^2v| dH^{n-1} \leq Ca_k^{\frac{3\gamma}{2} - 1} = Ct^{\frac{3\gamma}{2} - 1}(\frac{a_k}{t})^{\frac{3\gamma}{2} - 1} \leq Ct^{\frac{3\gamma}{2} - 1}. 
\end{equation*}
Hence, this implies that in fact $f(t) \leq Ct^{\gamma' - 1}$ where $\gamma' = \frac{\gamma}{2} + \frac{2}{3}$. Since the fixed point of the sequence $\{b_k\}$ defined recursively by $b_k = \frac{b_k}{2} + \frac{2}{3}$ and $b_0 = \gamma - 1$ is $\frac{4}{3}$, we conclude that $f(t) \leq Ct^{\frac{1}{3}}$ by iteration.
\end{proof}    

We conclude this section by proving the sharp regularity for solutions in annular domains, which in combination with Proposition \ref{corollary radial symmetry lip} gives the proof of Corollary \ref{corollary radial}.
\begin{proof}[Proof of Corollary \ref{corollary radial}]
In virtue of Proposition \ref{corollary radial symmetry lip}, it only suffices to consider the case when $\Omega$ is an annulus.
\medskip

\noindent {\it Step 1.} We prove smoothness for $u$ outside of its dead core as well as optimal non-degeneracy.\\

Without loss of generality, let us take $\Omega = B_{R_2}\setminus \overline{B_{R_1}}$ with $R_2>R_1>0$. Lemma \ref{lemma map general symmetry} guarantees that the corresponding solution $u$ of \eqref{main problem grad} is radially symmetric and is such that $u(x) = \zeta(|x|)$ for some $\zeta : [R_1,R_2]\to [0,\infty)$. Furthermore, arguing as in the proof of Proposition \ref{corollary radial symmetry lip}, we find that $u$ attains its global maximum in an annulus $D_u=\overline{B_{r_2(\alpha)}}\setminus B_{r_1(\alpha)}$ with $|D_u|=\alpha$ and with $R_1<r_1(\alpha)<r_2(\alpha)<R_2$. On the other hand, since $-\Delta u < 0$ in the viscosity sense in $\Omega \setminus D_u$, $u$ cannot attain any local minimizer outside of $D_u$. Hence, the level sets of $u$ are rings and $D_u = \{x \in \Omega\, | \nabla u(x) =0\}$, which implies smoothness for the right-hand side of \eqref{main problem grad} via Lemma \ref{lemma abs cont} and in turn that $u \in C^{\infty}_{loc}(\Omega \setminus N_u)$.\\
	
	Additionally, since $g'(t)\geq \frac{1}{C}$, the radial symmetry of $u$ implies $g(|u\geq u(x)|)\geq C(R_1(\alpha)-|x|)$.  Thus, we deduce from \eqref{main problem grad} the inequality	
	\begin{equation}\label{eq diff inequality}
		-(r^{n-1}\zeta')'  \geq \frac{1}{C} r^{n-1} (R_1(\alpha)-r)
	\end{equation} 
	for $r \in [R_1, R_1(\alpha)]$. Since $\zeta'(R_1(\alpha))=0$, integrating on $[s, R_2(\alpha)]$ for $s \in  [R_1, R_1(\alpha)]$ yields	
	\begin{equation}\label{eq diff inequality 1}
		s^{n-1}\zeta'(s)  \geq \frac{1}{C} \int_{s}^{R_1(\alpha)}r^{n-1} (R_1(\alpha)-r)dr\geq \frac{1}{C} s^{n-1}(R_1(\alpha)-s)^2.
	\end{equation}	
	By integrating one more time on $[s, R_2(\alpha)]$ we deduce that	
	\begin{equation}\label{eq left cubic bound}
		\max u -\zeta(r)\geq \frac{1}{C} (R_1(\alpha)-r)^3,\qquad \text{on} \quad [R_1, R_1(\alpha)].
	\end{equation}
	By a similar reasoning, we infer that	
	\begin{equation*}
		\max u -\zeta(r)\geq \frac{1}{C}(r-R_2(\alpha))^3,\qquad \text{on} \quad [R_2(\alpha), R_2].
	\end{equation*} 
	Altogether, we have found that
     \begin{equation}\label{eq optimal nondeg}
        \max u -u \geq \frac{1}{C} \text{dist}(x, N_u)^3.
    \end{equation}\\

\medskip

 \noindent {\it Step 2.} In this step, we prove  the following claim: if $g(|u\geq u(x)|)\leq C \text{dist}(x, [R_1(\alpha), R_2(\alpha)])^\beta$, for some $\beta \in [0,1]$, then	
	\begin{equation}\label{eq auxiliar diff ineq}
		\max u - \zeta(r)\leq C K(\beta) \text{dist}(x, [R_1(\alpha), R_2(\alpha)])^{\beta+2},
	\end{equation}
	with $K(\beta) \in (0,1)$ only depending on $\Omega$, $\beta$, and $\alpha$.\\
		
	To prove \eqref{eq auxiliar diff ineq}, we proceed in a similar vein as in \eqref{eq diff inequality}, deducing first that	
	\begin{equation}\label{eq diff inequality below}
		-(r^{n-1}\zeta')'  \leq C r^{n-1} (R_1(\alpha)-r)^\beta
	\end{equation} 
	for $r \in [R_1, R_1(\alpha)]$. Since $\zeta'(R_1(\alpha))=0$, we deduce integrating on $[s, R_2(\alpha)]$ for $s \in  [R_1, R_1(\alpha)]$
	
	\begin{equation}\label{eq diff inequality 1 below}
		s^{n-1}\zeta'(s)  \leq C \int_{s}^{R_1(\alpha)}r^{n-1} (R_1(\alpha)-r)dr\leq \frac{C}{1+\beta} R_1(\alpha)^{n-1}(R_1(\alpha)-s)^{1+\beta}.
	\end{equation}
	Integrating again on $[s, R_2(\alpha)]$, we deduce that
	
	\begin{equation}\label{eq left cubic bound below}
		\max u -\zeta(r)\leq \frac{C}{(1+\beta)(2+\beta)} \left(\frac{R_1(\alpha)}{R_1} \right)^{n-1} (R_1(\alpha)-r)^{2+\beta},\qquad \text{on} \quad [R_1, R_1(\alpha)].
	\end{equation}
	
	By a similar reasoning, we infer that
	
	\begin{equation}\label{eq left cubic bound below 2}
		\max u -\zeta(r)\leq \frac{C}{(1+\beta)(2+\beta)} \left(\frac{R_2}{R_2(\alpha)} \right)^{n-1} (r-R_2(\alpha))^{2+\beta},\qquad \text{on} \quad [R_2(\alpha), R_2].
	\end{equation}
		
	\medskip
    
	\noindent {\it Step 3.} In this step, we prove that if	
	\begin{equation}\label{eq left quad bound}
		\max u -\zeta(r)\leq C d(r, [R_1(\alpha), R_2(\alpha)])^{2+\beta},\qquad \text{on} \quad [R_1, R_2],
	\end{equation}	
	for some $\beta \in [0,1]$, then	
	\begin{equation}\label{eq g bound}
		g(|u\geq u(x)|)\leq 2\text{Lip}(g) C^\frac{2}{3}\text{dist}(x, [R_1(\alpha), R_2(\alpha)])^\frac{2+\beta}{3}, \qquad \text{on} \quad [R_1, R_2].
	\end{equation}
	
	Indeed, let $M= [R_1,R_2]\setminus [R_1(\alpha), R_2(\alpha)]$. Then, by combining \eqref{eq optimal nondeg} with \eqref{eq left quad bound} yields
	\begin{align*}
		|u \geq u(x)|-\alpha&= |y \in M\, | \max u -u(y) \leq \max u - u(x)|\\
		&\leq |y \in M\, | \frac{\text{dist}(y, [R_1(\alpha), R_2(\alpha)])^3}{C} \leq C\text{dist}(x, [R_1(\alpha), R_2(\alpha)])^{2+\beta}|\\
		&\leq |y \in M\, | \text{dist}(y, [R_1(\alpha), R_2(\alpha)]) \leq C^\frac{2}{3}\text{dist}(x, [R_1(\alpha), R_2(\alpha)])^\frac{2+\beta}{3}|\\
		&= 2 C^\frac{2}{3} \text{dist}(x, [R_1(\alpha), R_2(\alpha)])^\frac{2+\beta}{3}.
	\end{align*}
	
	Consequently, since $g$ is Lipschitz, we deduce \eqref{eq g bound}.\\

    \medskip
    
	\noindent {\it Step 4.} We iterate step 2 and step 3 to conclude the result.\\

    Set $L= 2 Lip(g)$, $K=K(\beta)$, $C_0= C$, and $\beta_0=0$. By combining step 2 and step 3, we have that
  \begin{equation}\label{eq interative bound}
		\max u -\zeta(r)\leq C_n d(r, [R_1(\alpha), R_2(\alpha)])^{2+\beta_k},\qquad \text{on} \quad [R_1, R_2],
	\end{equation}	
with $\beta_k = \frac{\beta_{k-1}+2}{3}$ and $C_k = L(C_{k-1} K)^\frac{2}{3}$. It is not difficult to see that the sequence $\{\beta_k\}_{k\in \N}$ is increasing and therefore converges to $1$, whereas, inductively, we have that
\begin{equation*}
    C_k = L (LC K)^{\sum_{i=1}^k \Big(\frac{2}{3}\Big)^i},
\end{equation*}
which shows the uniform boundedness of $\{C_k\}_{k\in \N}$. Thus, by taking limit as $k\to \infty$, we deduce 
\begin{equation*}
		\max u -\zeta(r)\leq C d(r, [R_1(\alpha), R_2(\alpha)])^{3},\qquad \text{on} \quad [R_1, R_2].
	\end{equation*}
The rest of the properties of $g(|u\geq t|)$ and $u$ readily follow, arguing as in Proposition \ref{corollary radial symmetry lip}.	
\end{proof}

\subsection{From semilinear equations to one-phase free boundary problems}\label{subsect semilinear}

The final part of the work is devoted to proving Theorem \ref{thm conditional}. In this section, we will work with a general semilinear equation that has a continuous and increasing right-hand side. Letting $\Omega$ be an open domain, we consider the problem 
\begin{equation}\label{eq semilinear}
\begin{cases}
    \Delta v = f(v) \qquad \text{ in } \{v>0\}\\
|\nabla v| = 0 \qquad \text{ on } \partial \{v>0\},
\end{cases}
\end{equation}
with $f$ strictly increasing, continuous, $f(0) = 0$, and satisfying assumptions $(A_1)-(A_3)$.\\

Notice that in the case of the Grad equation, Proposition \ref{lemma nondegeneracy} gives that $F(t) \geq \frac{1}{C}t^\frac{4}{3}$, a fact which guarantees that the integral in the definition of $h$ is convergent. As explained in the introduction, we can use $h$ to transform  the semilinear equation into a new problem where the free boundary condition appears explicitly. More precisely, see that $w=h(v)$ satisfies the degenerate one-phase problem
\begin{equation}\label{eq one phase}
\begin{cases}
        \Delta w(x) = a(x)\Bigg(\frac{2-|\nabla w(x)|^2}{w}\Bigg), \quad \mbox{in $\{w >0\}$}\\
        |\nabla w|^2=2, \hspace{3.1cm} \mbox{on $FB(w) := \partial \{w >0\}$}
\end{cases}
\end{equation}
where 
\begin{equation}\label{eq coefficient}
   a(x) =  \frac{f(v(x)) h(v(x))}{2\sqrt{F}(v(x))}.
\end{equation}

Whereas the interior equation in \eqref{eq one phase} will hold in the classical sense, the free boundary condition must be understood in the viscosity sense, as in \cite{de2021certain}, which we outline for the convenience of the reader.
\begin{definition}\label{def viscosity fbc}
    Let $C>0$. We say that $|\nabla w| \leq C$ in the viscosity sense on $FB(w)$ if given $x_0 \in FB(w)$ and $\psi\in C^2$ such that $\nabla \psi(x_0) \neq 0$ and $\psi^+$ touches $w$  from below at $x_0$ (i.e., $\psi_+ \leq w$ on a neighborhood of $x_0$ with $\psi(x_0)=0$), then $|\nabla \psi(x_0)|\leq C$.\\

    Analogously, we say that $|\nabla w| \geq C$ in the viscosity sense on $FB(w)$ if given $x_0 \in FB(w)$ and $\psi\in C^2$ such that $\nabla \psi(x_0) \neq 0$ and $\psi^+$ touches $w$  from above at $x_0$ (i.e., $\psi_+ \geq w$ on a neighborhood of $x_0$ with $\psi(x_0)=0$), then $|\nabla \psi(x_0)|\geq C$.
\end{definition}

Our next step is to construct a family of barriers that will allow us to verify the free boundary condition in \eqref{eq one phase}.

\begin{lemma}\label{eq radial barriers}
 Given $r > 0$ we can find $R=R(r)$ so that the problems
 \begin{equation}\label{eq increasing sol}
     \begin{cases}
         \Delta \overline{v}_r = f(\overline{v}_r), \quad B_R\setminus \overline{B_r},\\
         \overline{v}_r = 1, \quad \partial B_r,\\
         \overline{v}_r = 0, \quad \partial B_R,
     \end{cases}
 \end{equation}
 and
  \begin{equation}\label{eq decreasing sol}
     \begin{cases}
         \Delta \underline{v}_r = f(\underline{v}_r), \quad B_R\setminus \overline{B_r},\\
         \underline{v}_r = 0, \quad \partial B_r,\\
         \underline{v}_r = 1, \quad \partial B_R,
     \end{cases}
 \end{equation}
admit unique solutions. Moreover, $\overline{v}_r, \underline{v}_r \in C^{1,1}(\overline{B_R\setminus B_r})$, are radial functions, and satisfy
\begin{equation}\label{eq upper barrier}
    h(\overline{v}_r) \leq \sqrt{2}(R-|x|), 
\end{equation}
and
\begin{equation}\label{eq lower barrier}
    h(\underline{v}_r) \geq \sqrt{2}(|x|-r).
\end{equation}

\end{lemma}
\begin{proof}
The construction of both solutions follows from the method of monotone iterations. In both cases, we start by noticing that the non-trivial one dimensional solution of \eqref{eq one dimensional} given by $U(s)= h^{-1}(\sqrt{2}s)$ is strictly increasing and, in consequence, there exists a unique $\kappa>0$ such that $U(\kappa)=1$. Let us set $R=r+\kappa$ and let us notice that $v_{sup}(x) = U(R - |x|)$ satisfies
\begin{equation}\label{eq supersol}
    \Delta v_{sup}(x) = U''(R-|x|) - \frac{n-1}{|x|}U'(R - |x|) \leq f(U(R-|x|)) = f(v_{sup}(x)),
\end{equation}
whilst $v_{sub}(x) = U(r + |x|)$
\begin{equation}\label{eq subsol}
    \Delta v_{sub}(x) = U''(r+|x|) + \frac{n-1}{|x|}U'(r + |x|) \geq f(U(r +|x|)) = f(v_{sub}(x)).
\end{equation}
From \eqref{eq supersol} we have that $v_{sup}$ is a supersolution to \eqref{eq increasing sol}, while \eqref{eq subsol} implies that $v_{inf}$ is a subsolution to \eqref{eq decreasing sol}.\\

Now, we define a sequence of functions $v_{k}$ starting with $v_0 = v_{sup}$ and taking, inductively, $v_{k}$ as the unique solution to the linear problem
 \begin{equation}\label{eq increasing iter}
     \begin{cases}
         \Delta v_k = f(v_{k-1}), \quad B_R\setminus \overline{B_r},\\
         v_k = 1, \quad \partial B_r,\\
         v_k = 0, \quad \partial B_R.
     \end{cases}
 \end{equation}
Since $f$ is strictly increasing, a direct application of the maximum principle implies that $v_{k+1} \leq v_k$ $B_R\setminus \overline{B_r}$. From here, it follows from standard elliptic regularity theory that $v_{k} \rightarrow \overline{v_r}$ as $k\to \infty$ in $H^1(B_R\setminus B_r)\cap C(\overline{B_R\setminus B_r})$ with $\overline{v_r}$ satisfying \eqref{eq increasing sol} in the sense of distributions. Moreover, $\overline{v_r} \in C^{1,1}(\overline{B_R\setminus B_r})$, see \cite{shahgholian2003c1}.\\

Now, we show uniqueness of \eqref{eq increasing sol}; uniqueness for \eqref{eq decreasing sol} follows the same lines. Suppose that we have two solutions $v_1$ and $v_2$ for $\eqref{eq increasing sol}$. By way of contradiction, we assume that $v_1 - v_2$ attains a positive maximum at $x_0$. Then, we have that
\begin{equation*}
0 \geq \Delta(v_1 - v_2)(x_0) = f(v_1(x_0)) - f(v_2(x_0)) > 0,
\end{equation*}
a contradiction. Hence, $v_1 \leq v_2$. By symmetry, we have that $v_2 \leq v_1$, and so $v_1 = v_2$. 

We now turn our attention to showing that the solution to \eqref{eq increasing sol} is a radial function; showing the corresponding result for \eqref{eq decreasing sol} is identical. Assuming $v$ is the unique solution to \eqref{eq increasing sol}, let $\tilde{v} = v(Ax)$, where $A$ is any rotation matrix. Then, it follows that for $x \in B_R \setminus B_r$,

\begin{equation*}
\Delta \tilde{v}(x) = \Delta v(Ax) = f(v(Ax)) = f(\tilde{v}(x)).    
\end{equation*}

Since $v = \tilde{v}$ on $\partial B_R$ and $\partial B_r$, it follows that $v = \tilde{v}$, and so $v$ must be a radial function.

Also, we note that since $\overline{v_r}  \leq v_{sup}$, by the strict monotonicity of $h$, we have that $h(\overline{v_r}) \leq h(v_{sup})$. Then, since  $h(U)(t) = \sqrt{2}t$ the bound \eqref{eq upper barrier} readily follows.\\

The existence of the solution for \eqref{eq decreasing sol} together with \eqref{eq lower barrier} follows from applying again the method of monotone iterations to $v_{sub}$.
\end{proof}

We are in a position to prove the validity of \eqref{eq one phase}.

\begin{proposition}\label{prop viscosity sol}
Let $v$ be a solution to \eqref{eq semilinear}. Then, $w= h(v)$ is a solution to the one-phase problem \eqref{eq one phase} satisfying the free boundary condition in the viscosity sense. 
\end{proposition}
\begin{proof}
Let us start verifying the interior equation in \eqref{eq one phase}. A direct computation shows that for a.e. $x \in \{w>0\}$
\begin{eqnarray}\notag
    \Delta w(x) &=& h''(v)|\nabla v|^2+h'(v)\Delta v \\ \label{eq Alt-Caffarelli}
    &=&  \frac{f(v(x)) h(v(x))}{2\sqrt{F}(v(x))}\Bigg(\frac{2-|\nabla w(x)|^2}{w}\Bigg).
\end{eqnarray}
The continuity of both terms in \eqref{eq Alt-Caffarelli} thus implies that the equation holds in the classical sense.\\

With our aim set on the free boundary condition, let us assume that $0 \in FB(w)$. Let $\psi$ be a test function as in Definition \ref{def viscosity fbc} that touches $v$ from below at $0$. By way of contradiction, assume that $|\nabla \psi(0)|^2 > 2$ and set $a = |\nabla \psi(0)|$ and $\nu = \frac{\nabla \psi(0)}{|\nabla \psi(0)|}$. By the $C^2$ regularity of $\psi$, given $\delta > 0$, we can find $\rho>0$ sufficiently small so that $w(x) \geq \psi_{+}(x) \geq (1 + \delta)\sqrt{2}(x \cdot \nu)_{+}$ in a ball $B_{\rho}(0)$.\\

Now, fix $R>1$ so that \eqref{eq increasing sol} has a solution on $B_R\setminus \overline{B_1}$, and let $v_{0}$ be such a solution. Then, for small $\theta > 0$,  define $v_{\theta}(x) = v_0(x - (R - \theta)\nu)$. Note that for all $\theta > 0$ we have that $v_{\theta}(0) > 0$. Likewise, define $w_{\theta} = h(v_{\theta})$. Let us choose $\theta$ so that $\rho^2 + (R - \theta)^2 = R^2$. This then guarantees that if $x \cdot \nu \leq 0$ and $|x| = \rho$, then $|x - (R - \theta)\nu| \geq R$, and thus $v_{\theta}(x)$ and $w_{\theta}(x) = 0$. On the other hand, from \eqref{eq upper barrier}, we have that $w_{\theta}(x) \leq \sqrt{2}(R-|x-(R-\theta)\nu|)$. Thus, we can find $\rho>0$ such that $w_{\theta}(x) \leq \big(1 + \frac{\delta}{2}\big)\sqrt{2}(x \cdot \nu)_+$ for $x \in \partial B_{\rho}$. \\

Then, we have shown that $w_{\theta} \leq w$ on $\partial B_{\rho}$. This implies that $v_{\theta} \leq v$ on $\partial B_{\rho}$, yet $v_{\theta}(0) > v(0) = 0$. This is a contradiction. Indeed, under these circumstances, $v_{\theta} - v$ attains a positive maximum in the interior of $B_{\rho}$ at some point $x_*$, which thanks to the maximum principle, implies that 
\begin{equation*}
    0 \leq \Delta(v_{\theta}-v)(x_*) = f(v_{\theta})(x_*)-f(v)(x_*)
\end{equation*}
and thus, by the strict monotonicity of $f$, we have that $ v(x_*)-v_{\theta}(x_*)\leq 0$.\\

The other side of the viscosity free boundary condition follows analogously, using instead the function \eqref{eq decreasing sol} as a barrier.
\end{proof}

Now, we notice that the one-phase problem \eqref{eq one phase} fall into the type of equations studied by De Silva and Savin in \cite{de2021certain}. In such work, the authors investigate regularity for solutions to equations of the form
\begin{equation}\label{eq DSS}
    \Delta w = \frac{\Phi(\nabla w)}{w} \mbox{ in $\{w>0\}$ },
\end{equation}
coupled with suitable free boundary conditions on $\partial\{w>0\}$. They show that if $\Phi$ is bounded, solutions to \eqref{eq DSS} are Lipschitz continuous up to the boundary of $\{w>0\}$. In the context of our equation, the boundedness of $\Phi$ is guaranteed by $(A_2)$, as we now show.  

\begin{lemma}\label{lemma cond conseq}
    If $f$ satisfies $(A_2)$, i.e., 
    \begin{equation}
        \lim_{t \to 0^+} \frac{tf(t)}{F(t)} =\omega \in [1,2),
    \end{equation}
    then
    \begin{equation}\label{eq coefficient rescaling}
    \lim_{t\to 0^+} \frac{f(t)h(t)}{\sqrt{F(t)}} =  \frac{2 \omega}{2-\omega} . 
\end{equation}
\end{lemma}
\begin{proof}
    We can use L'Hopital's rule together with $(A_2)$ to deduce
\begin{equation}\label{eq equivalence h}
    \lim_{t\to 0^+} \frac{t \frac{1}{\sqrt{F(t)}}}{h(t)} = \lim_{t\to 0^+} \frac{t \frac{1}{\sqrt{F(t)}}}{h(t)} =  \lim_{t\to 0^+}  \frac{ \frac{1}{\sqrt{F(t)}} - \frac{ t f(t)}{2 F(t)^\frac{3}{2}}}{\frac{1}{\sqrt{F(t)}}} = 1 - \frac{\omega}{2}.
\end{equation}

Therefore, by combining \eqref{eq equivalence h} and $(A_2)$ we conclude that
\begin{equation}\label{eq coefficient rescaling}
    \lim_{t\to 0^+} \frac{f(t)h(t)}{\sqrt{F(t)}} =  \frac{2}{2-\omega} \lim_{t\to 0^+}  \frac{f(t)t}{F(t)} =  \frac{2 \omega}{2-\omega}.
\end{equation}
\end{proof}

\begin{lemma}\label{lemma gradient estimate}
Let $v$ be a solution of \eqref{eq semilinear}. Then, $w=h(v)$ is uniformly Lipschitz in $\Omega$. In particular, $|\nabla v|^2 \leq CF(v)$. 
\end{lemma}
\begin{proof}
Proposition \ref{prop viscosity sol} and Lemma \ref{lemma cond conseq} imply that $w$ solves an equation of the form 
\begin{equation}\label{eq one phase kind of}
\begin{cases}
        \Delta w(x) = \Phi(x,\nabla w), \quad \mbox{in $\{w >0\}$}\\
        |\nabla w|^2=2, \hspace{3.1cm} \mbox{on $\partial \{w >0\}$}.
\end{cases}
\end{equation}
with 
\begin{equation*}
   -C|p|^2 \leq \Phi(x, p) \leq C\chi_{B_{\sqrt{2}}}
\end{equation*}
uniformly in $x$. Then, we can invoke directly Theorem 3.6 in \cite{de2021certain} to deduce that $|\nabla w|\leq C$ in $\Omega$ which can be rewritten as $|\nabla v|^2 \leq CF(v)$.
\end{proof}

Our next step is to quantify the gradient bound obtained in the previous result near the free boundary. We follow the lines of reasoning of Phillips in \cite{phillips1983hausdoff}. In the following lemma, our regularity assumptions on $\Phi$ are weaker than those in \cite{de2021certain}.

\begin{lemma}\label{lemma optimal gradient bound}
    Let $v$ be a solution of \eqref{eq semilinear}, and let $x_0 \in FB(v)$. Given $\eta>0$, there exists $r>0$ such that
      \begin{equation}\label{eq gradient bound enhanced}
    |\nabla v|^2 \leq (2+\eta) F(v), \quad \mbox{ on $B_r(x_0) \cap \{v > 0\}$}.
\end{equation}
\end{lemma}
\begin{proof}
    Let $x\in \partial D$ 
    \begin{equation}\label{eq minimum}
       \psi(\e)= \sup_{\{u>0\}\cap B_\e(x)} \frac{|\nabla v|^2}{F(v)}.
    \end{equation}

The proof boils down to showing that $\psi(0^+)=\lim_{\e\to 0} \psi(\e) \leq 2$. So, we argue by contradiction assuming $\psi(0^+)>2$. We take, without loss of generality, $x=0$. Then, there exists a sequence $y_k \in B_{1/k}\cap \{v>0\}$ such that
\begin{equation}\label{eq contradiction sequence b}
    \frac{|\nabla v|^2(y_k)}{F(v)(y_k)} =M_k \to  \psi(0^+).
\end{equation}

The contradiction statement \eqref{eq contradiction sequence b} can be phrased in terms of $w$ as
\begin{equation}\label{eq contradiction sequence}
    |\nabla w|^2(y_k) =M_k \to  \psi(0^+).
\end{equation}

We consider the rescalings $w_k(x) = \frac{w(y_k +\rho_k x)}{\rho_k}$ with $\rho_k= w(y_k)$. Notice that since $w$ is Lipschitz continuous, Lip$(w_k)=$Lip$(w)$, implying that the functions $w_k$ are uniformly Lipschitz continuous. Also, since $w_k(0)=1$ we have that $w_k(x) \geq \frac{1}{2}$ for $x \in B_\frac{1}{C}$ for some constant only depending on the Lipschitz constant of $w$. Additionally, up to a subsequence, $w_k \to w_0$ uniformly in $B_1$ with $w_0$ Lipschitz in $B_1$ and satisfying Lip$(w_0)=$Lip$(w)$.  

Hence, we deduce from \eqref{eq Alt-Caffarelli trans}, \eqref{eq coefficient rescaling} and the uniform convergence of $w_k$ that
\begin{equation}\label{eq Alt-Caffarelli trans}
    \Delta w_k(x) = a_k(x)\Bigg(\frac{2-|\nabla w_k(x)|^2}{w_k}\Bigg), \mbox{in $\{u >0\}$}.
\end{equation}
where 
$$a_k(x) =  \frac{f(v(y_k+\rho_kx)) h(v(y_k+\rho_kx))}{2\sqrt{F}(v(y_k+\rho_kx))} \to \lambda$$ 
locally uniformly thanks to Lemma \ref{lemma cond conseq}. In particular, $\Delta w_k$ is a bounded function, which implies uniform $C^{1,\alpha}$ bounds for $w_k$ in $B_1$ and thus that $\nabla w_k \to \nabla w_0$ in $B_1$, in particular $|\nabla w_0| >2$ and, from the definition of $\psi$,
\begin{equation}\label{eq maximality gradient}
    |\nabla w_0(x)| \leq |\nabla w_0(0)| \quad \mbox{ for $x\in B_1$ }.
\end{equation}

On the other hand, we have that
\begin{equation}\label{eq limiting equation}
    \Delta w_0(x) = \lambda\frac{2-|\nabla w_0(x)|^2}{w_0} \quad \mbox{for $x\in B_1$},
\end{equation}
which, thanks to the Lipschitz regularity of $w_0$, yields smoothness for $w_0$. We proceed noticing that, in virtue of \eqref{eq maximality gradient}, we have that $\Delta |\nabla w_0|^2(0)\leq 0$. Our contradiction will follow from showing that the opposite inequality also holds. Combining Bochner's formula and \eqref{eq limiting equation} yields
\begin{eqnarray}\notag
      \Delta |\nabla w_0|^2 &=& 2 | D^2 w_0|^2 + 2 \nabla w_0\cdot \nabla \Delta w_0\\ \label{eq lower bound lap norm}
      &\geq& 2 \nabla w_0\cdot \nabla \Delta w_0.
\end{eqnarray}
We estimate now the right hand side of  \eqref{eq lower bound lap norm} at $0$. Multiplying on both sides of \eqref{eq limiting equation} by $w_0$, differentiating and using the optimality condition $\nabla |\nabla w_0|^2(0)=0$ together with $w_0(0)=1$, we deduce
\begin{equation}\label{eq gradient laplacian}
    \nabla w_0\cdot \nabla \Delta w_0(x) = -|\nabla w_0|^2(0) \Delta w_0(0)= \lambda (|\nabla w_0(x)|^2-2)|\nabla w_0(x)|^2.
\end{equation}

Finally, \eqref{eq lower bound lap norm} and \eqref{eq gradient laplacian} altogether imply $\Delta  |\nabla w_0|^2(0)>0$, which is a contradiction.
\end{proof}

We now show that if $v$ solves \eqref{eq semilinear}, then $w = h(v)$ is a solution to \eqref{eq one phase} that satisfies a non-degeneracy bound.

\begin{proposition}
Let $v$ be a solution to \eqref{eq semilinear} in $\Omega$ with $0 \in FB(v)$, and let $w = h(v)$. Then, it follows that for $\rho \leq \frac{1}{C}$
\begin{equation}\label{eq nondegeneracy}
    \sup_{x \in B_\rho} w(x) \geq \frac{1}{C}\rho. 
\end{equation}
\end{proposition}
\begin{proof}
For $\delta > 0$ to be chosen later, take $\rho$ small enough so that $|\nabla w|^2 \leq 2 + \delta$ on $B_{\rho}$. Then, we compute that

\begin{equation*}
\Delta w^2 = 2|\nabla w|^2 + a(x)[2 - |\nabla w|^2] = (2 - a(x))|\nabla w|^2 + 2a(x).
\end{equation*}
The above quantity is bounded above by a positive constant if $a(x) \leq 2$. If $a(x) > 2$, then

\begin{equation*}
\Delta w^2 \geq (2 - a(x))(2 + \delta) + 2a(x) = 4 + 2\delta - \delta a(x).
\end{equation*}

The above quantity will be positive by choosing $\delta \leq \frac{4}{\sup_{x \in \Omega}a(x)}$. Hence, $\Delta w^2 > \frac{1}{C}$. 

Consequently, it follows that $w_0 \equiv w^2 - \frac{1}{C}|x|^2$ is a subharmonic function such that $w_0(0) = 0$. Hence, there exists some $x_0 \in \partial B_{\rho} \cap \{w > 0\}$ such that $w_0(x_0) \geq 0$. Therefore, $w^2(x_0) \geq \frac{1}{C}\rho^2$, and the proof is complete. 
\end{proof}

Notice that, as a corollary of Lemma \ref{lemma gradient estimate}, solutions have a linear growth away from the free boundary; that is,

\begin{equation}\label{eq linear growth}
w(x) \leq Cdist(x, FB(w)).
\end{equation}

Combining \eqref{eq linear growth} with \eqref{eq non-degeneracy}, we obtain the uniform density of the positivity set along the free boundary.
\begin{corollary}\label{corollary uniform density}
If $0 \in FB(w)$, then given $r > 0$, there exists $x_0$ such that $B_{\frac{r}{c}}(x_0) \subset B_r \cap \{w > 0\}$.
\end{corollary}

The next lemma precises the asymptotic convergence of the previous gradient bound as we approach the free boundary. 

\begin{lemma}\label{lemma supergradient bound 3000}
For $s > 0$, let $l_s(t) = \int_0^t F(l)^{s}dl$. Then, there exists $\delta > 0$ and an $s \in (0, \frac{1}{2})$ such that for any $x_0 \in FB(v)$,
    \begin{equation}\label{eq improved gradient bound}
        |\nabla v|^2 \leq 2F(v)+ v^{1-2s}l_s(v) \quad \mbox{ in $B_{\delta}(x_0)$}.
    \end{equation}
\end{lemma}
\begin{proof}
Choose $x \in \{v > 0\} \cap B_{\delta}(x_0)$ and let $\tau$ be such that $B_{3\tau}(x) \subset \{v > 0\} \cap B_{\delta}(x_0)$. Then, we define $\zeta(t) = \kappa(t - \tau)_{+}^{3}$ for $\kappa > 0$ to be chosen later.
With this in mind, we let
\begin{equation}\label{eq w bar}
\bar{w}(y) = |\nabla v(y)|^2 - 2F(v(y)) - v^{1-2s}l_{s}(v(y)) - \zeta(|y-x|)l_1(v(y)),
\end{equation}
where $s$ is to be chosen later. We choose $\kappa$ large enough in \eqref{eq w bar} so that $\bar{w} \leq 0$ on $\partial B_{3\tau}(x)$. By way of contradiction, we will now assume that $\bar{w}$ attains a positive maximum at some $y_0 \in B_{3\tau}(x)$ and show that $\Delta \bar{w} > 0$ at such point, yielding a contradiction. If this is the case, since $\zeta \equiv 0$ on $B_{\tau}(x)$, we will have that \eqref{eq improved gradient bound} holds at $x$, thus proving the result. \\

Without loss of generality, we take $y_0 = 0$ for the rest of the proof. By a rotation of coordinates, we may assume that $\nabla v(0)$ is collinear with $e_n$ at such point. Then, using the fact that $\bar{w}$ obtains a positive maximum at $0$, we take a derivative in the direction $e_n$ in \eqref{eq w bar} at the origin to get that 

\begin{eqnarray}\nonumber
v_{nn}(0) &=&  f(v(0)) + \frac{1}{2}[[(1-2s)v(0)^{-2s}l_s(v(0)) + F^{s}(v(0))v(0)^{1-2s}] ] \\ \label{eq xn derivative}
&&+ \frac{3\kappa(|x| - \tau)_{+}^{2}x_nl_1(v(0))}{2|x|v_n(0)} + \frac{\kappa(|x| - \tau)_{+}^{3}F(v(0))}{2}.
\end{eqnarray}
On the other hand, Bochner's formula yields 
\begin{equation}\label{eq bochner}
\Delta |\nabla v|^2 = 2f'(v)|\nabla v|^2 + 2|D^2v|^2,
\end{equation}
while a direct computation yields
\begin{equation}\label{eq F}
\Delta 2F(v) = 2f'(v)|\nabla v|^2 + 2f^2(v),
\end{equation}
\begin{eqnarray} \nonumber
\Delta [v^{1-2s}l_s(v)] &=& (1-2s)v^{-2s}l_s(v)f(v) + v^{1-2s}f(v)F^{s}(v) \\ \nonumber
&&- (1-2s)(2s)v^{-1 - 2s}l_s(v)|\nabla v|^2 + (1-2s)v^{-2s}F^{s}(v)|\nabla v|^2 \\ \label{eq 3} 
&&+ sF^{s-1}f(v)v^{1-2s}|\nabla v|^2 + (1-2s)F^{s}(v)v^{-2s}|\nabla v|^2,   
\end{eqnarray}
and
\begin{eqnarray}\nonumber 
\Delta \zeta l_1(v) &=&  \frac{3(n-1)(|x| - \tau)_{+}^2l_1}{|x|} + 6\kappa(|x|-\tau)_{+}l_1 + \frac{6\kappa(|x|-\tau)_{+}^2(x)\cdot\nabla v F(v)}{|x|} \\ \label{eq 4}
&& + \kappa (|x| - \tau)_{+}^3f(v)|\nabla v|^2 + \kappa(|x| - \tau)_{+}^{3}F(v)f(v).
\end{eqnarray}

Then, using the fact that $2|D^2v|^2 \geq 2v_{nn}^{2}$ and invoking \eqref{eq xn derivative}, \eqref{eq bochner}, \eqref{eq F}, \eqref{eq 3}, and \eqref{eq 4}, we get that

\begin{equation*}
\Delta \bar{w}(0) \geq A(0) + B(0) + C(0),
\end{equation*}

where 

\begin{eqnarray*}
A &=& 2v^{1-2s}f(v)F^{s}(v) - sv^{1-2s}f(v)F(v)^{s-1}|\nabla v|^2, \\
B &=& \frac{1}{2}v^{2-4s}F^{2s} - 2(1-2s)v^{-2s}F^{s}|\nabla v|^2, \\
C &=& \frac{\kappa^2}{2}(|x| - \tau)_{+}^{3}F(v) - \kappa (|x| - \tau)_{+}^3f(v)|\nabla v|^2 - \kappa(|x| - \tau)_{+}^{3}F(v)f(v)\\
&&- 6\kappa(|x| - \tau)_{+}^2|\nabla v|F(v) -  \frac{3(n-1)(|x| - \tau)_+^2l_1(v)}{|x|} - \frac{6\kappa(|x|- \tau)^2F(v)x_n|\nabla v|}{|x|} \\
&& -6\kappa(|x|-\tau)_{+}l_1(v).
\end{eqnarray*}
By taking $s$ sufficiently close to $\frac{1}{2}$ and $x$ sufficiently close to the free boundary, and by invoking $(A_3)$, we ensure $A$ and $B$ are positive. Indeed, for any $s \in (0,\frac{1}{2})$,  we can take $\delta > 0$ such that, by Lemma \ref{lemma optimal gradient bound}, it follows that
\begin{equation*}
A(0) \geq v(0)^{1 - 2s}f(v(0))F(v(0))[1 - s(2 + \eta(\delta))] > 0.
\end{equation*}
Analogously, by using again Lemma \ref{lemma optimal gradient bound} and hypothesis $(A_3)$, we can find $s$ sufficiently close to $\frac{1}{2}$ (depending on $\epsilon$) and $\delta > 0$ small enough such that 
\begin{eqnarray*}
B(0) &\geq& \frac{1}{4}v(0)^{2-4s}F^{2s} - 2(1-2s)(2 + \eta(\delta))v(0)^{-2s}F(v(0))^{s+1} \\
&&\geq F(v(0))^{2s}\big(\frac{1}{4}v(0)^{2-4s} - \frac{2}{(1 + \epsilon)^{1-s}}(1 - 2s)(2 + \eta(\delta))(v(0)^{(\epsilon + 1)(1-s)})\big) > 0.
\end{eqnarray*}
Further, by choosing $\kappa$ large enough, we ensure $C(0)$ is positive as well. This yields a contradiction.
\end{proof}

Notice that since $F(t)$ is an increasing function, then $l_s(v) \leq vF^{s}(v)$. Hence, the previous estimate implies that there is a neighborhood of the free boundary in which

\begin{equation}\label{improved gradient bound}
|\nabla w|^2 - 2 \leq \frac{v^{1-2s}l_s(v)}{F(v)} \leq \big(\frac{v}{\sqrt{F(v)}}\big)^{2-2s}\leq w^{2-2s},
\end{equation}
where we have used that $w \geq \frac{v}{\sqrt{F(v)}}$, thanks to the monotonicity of $F$. \\

We now use \eqref{improved gradient bound} to show an integrability estimate near the free boundary. The proof and the result are the same as in Lemma 4.5 of \cite{de2021certain}, but we include them for completeness. Without loss of generality, we assume now that the origin is a free boundary point and that the equation holds in $B_1$.

\begin{lemma}\label{lemma integrability w-1}
Under the assumptions of the previous theorem, we have that
\begin{equation}
\int_{B_1 \cap \{w > 0\}} \frac{|\nabla w|^2 - 2}{w} \leq C.
\end{equation}
\end{lemma}
\begin{proof}
Letting $s$ be as in the previous result, for $\theta > 0$ and $\xi = 2 - 2s > 1$ we let $m(t)$ be a $C^{1,1}$ smoothing of $(t^{1 + \xi})_+$; that is, $m$ satisfies $m(0) = m'(0) = 0$ and $m''(t) = \min\big\{\frac{1}{\theta}, t^{\xi - 1}\big\}$. Then, using \eqref{eq one phase}, we compute that
\begin{equation*}
\Delta m(w) = m''(w)\Big(\frac{m'(w)}{m''(w)w}a(x)[|\nabla w|^2 - 2] + |\nabla w|^2\Big).
\end{equation*}
We can take $a(x)(|\nabla w |^2 - 2)$ sufficiently small by Lemmas \ref{lemma cond conseq} and \ref{lemma optimal gradient bound} so that
\begin{equation*}
\Delta m(w) \geq (1 - \frac{1}{C}m''(w))|\nabla w|^2.
\end{equation*}
Now, the divergence theorem along with the fact that $w$ is Lipschitz gives that
\begin{equation*}
\int_{\{w > 0\} \cap B_1} \Delta m(w) \leq C.
\end{equation*}

Hence, we can conclude that 

\begin{equation*}
\int_{\{w > 0\}} m''(w) \leq C.
\end{equation*}

But, by Lemma \ref{lemma supergradient bound 3000}
\begin{equation*}
\frac{|\nabla w|^2 - 2}{w} \leq w^{\xi - 1} = \lim_{\theta \rightarrow 0} m''(w). 
\end{equation*}
The result then follows by letting $\theta \rightarrow 0$.
\end{proof}
We are now in a position to prove Theorem \ref{thm conditional}. The proof also follows the same line as in \cite{de2021certain}.

\begin{proof}
For $\theta > 0$ and $\xi = 2 - 2s$, let $m(t)$ be as in Lemma \ref{lemma integrability w-1}, but with support in $[\theta, 4\theta]$. Then, Lemma \ref{lemma integrability w-1} implies that 

\begin{equation}
C \geq \int_{B_1} \Delta m(w) \geq \int_{B_1} m''(w)|\nabla w|^2 + m'(w)a(x)(\frac{2 - |\nabla w|^2}{w}) \geq \int_{B_1 \cap \{w \in (\theta, 2\theta)\}} \frac{1}{\theta}|\nabla w|^2 - C.    
\end{equation}

Thus, we conclude that
\begin{equation}
|\theta < w < 2\theta \cap B_1| \leq C\epsilon. 
\end{equation}

Now, take a covering of the free boundary in $B_1$ by $N$ balls of radius $C\theta$ centered at some $z_k$. Then, it follows that

\begin{equation}
\sum_{k=1} ^ {N} |B_{C\theta}(z_k) \cap w > 0| \leq C\theta.
\end{equation}

However, due to Corollary \ref{corollary uniform density},
\begin{equation}
|B_{C\theta}(z_k) \cap w > 0| \geq C\theta^{n}.
\end{equation}

We thus get that 
\begin{equation}
NC\theta^{n-1} \leq C
\end{equation}
and the result follows.
\end{proof}

Finally, we show that if the right-hand side in \eqref{eq reversed problem} satisfies $(A2)$, then it also satisfies the optimal supersolution bound.

\begin{proof}[Proof of Corollary \ref{corollary grad cond}]
Assuming that $v$ is now the solution to \eqref{eq reversed problem}, we as usual let $w = h(v)$, with $0 \in FB(v)$. We have that for any $\theta, r > 0$
\begin{equation*}
\frac{1}{\theta}\int_{\{0 < u \leq \phi^{-1}(\theta)\} \cap B_{2r}} |\nabla \phi(v(x))|^2 dx = \frac{1}{\theta}\int_{B_{2r}} \nabla \phi(v(x)) \cdot \nabla \{\min(\phi(v(x)), \theta)\} dx. 
\end{equation*}
By a standard penalization argument, we are able to integrate the right-hand side by parts. This then implies that
\begin{eqnarray*}
I &:=& \frac{1}{\theta} \int_{\{0 < v \leq \phi^{-1}(\theta)\} \cap B_{2r}} |\nabla \phi(v)|^2 + \phi(v)\Delta \phi(v)\\ 
&=& -\int_{\{\phi^{-1}(\theta) < v\} \cap B_{2r}} \Delta \phi(v) + \int_{\partial B_{2r}} \frac{\min\{(\phi(v), \theta)\}}{\theta} (\phi(v)_{\nu}) dH_{n-1}. 
\end{eqnarray*}

Now, we compute that 

\begin{eqnarray*}
|\nabla \phi(v)|^{2} = \frac{|\nabla v|^2}{F(v)} \\
\Delta \phi(v) = -\frac{f(v)}{2F(v)^{\frac{3}{2}}}|\nabla v|^2 + \frac{f(v)}{F(v)^{\frac{1}{2}}}.
\end{eqnarray*}

And so
\begin{eqnarray}\label{eq combined}
|\nabla \phi(v)|^2 + \phi(v)\Delta \phi(v) &=& \frac{|\nabla v|^2}{F(u)} - \frac{\phi(v)f(v)|\nabla v|^2}{2F(v)^{\frac{3}{2}}} + \frac{\phi(v)f(v)}{F(v)^{\frac{1}{2}}} \nonumber \\ 
&=& \frac{|\nabla v|^2}{F(v)}\Big[1 - \frac{\phi(v)f(v)}{2F(v)^{\frac{1}{2}}}\Big] + \frac{\phi(v)f(v)}{F(v)^{\frac{1}{2}}}.
\end{eqnarray}
Then, let us seek an upper bound for the term $\int_{\{0 < v \leq \phi^{-1}(\theta)\}} \frac{|\nabla v|^2\phi(v)f(v)}{F(v)^{\frac{3}{2}}}$. Applying the coarea formula and denoting $t = \phi^{-1}(\theta)$, we find that
\begin{eqnarray*}
\int_{\{0 < v \leq t\}} \frac{|\nabla v|^2\phi(v)f(v)}{F(v)^{\frac{3}{2}}} \leq C\int_{0} ^ {t} \int_{v = s} \frac{|\nabla v|\phi(s)f(s)}{F(s)^{\frac{3}{2}}} dH_{n-1}ds = C\int_{0} ^ {t} \frac{\phi(s)f^{3}(s)}{F(s)^{\frac{3}{2}}},
\end{eqnarray*}
where we have used that $\int_{v = s} |\nabla v| \leq Cf(s)^{2}$ (see the proof of Lemma \ref{lemma nondegeneracy}, specifically \eqref{eq testinlevelsets}). Now, since by assumption there exists some $\delta > 0$ such that $F(s) \geq \delta sf(s)$, then it follows that
\begin{equation}\label{meow}
\int_{0} ^ {t} \frac{\phi(s)f^{3}(s)}{F(s)^{\frac{3}{2}}} \leq \frac{1}{\delta}\int_{0} ^ {t} \frac{\phi(s)f^{\frac{3}{2}}(s)}{s^{\frac{3}{2}}}.
\end{equation}
Now, we notice that due to the optimal subsolution lower bound on $f$ from Lemma \ref{lemma nondegeneracy}, 
\begin{equation*}
F(s) = \int_{0} ^ {s} f(l) dl \geq C\int_{0} ^ {s} l^{\frac{1}{3}} dl = Cs^{\frac{4}{3}},
\end{equation*}
which yields
\begin{equation}\label{eq phi}
\phi(s) = \int_{0} ^ {s} F(l)^{-\frac{1}{2}} \leq C\int_{0} ^ {s} l^{-\frac{2}{3}} = Cs^{\frac{1}{3}}.
\end{equation}

Now, assume that $f(t) \leq Ct^{\gamma - 1}$ for some $\gamma < \frac{1}{3}$ (this is true for $\gamma = \frac{1}{5}$, as was shown in Theorem \ref{theorem suboptimal supersol}). We plug this in along with \eqref{eq phi} in \eqref{meow} to find that 
\begin{equation*}
C\int_{\{0 < v \leq \phi^{-1}(\theta)\}} \frac{|\nabla v|^2\phi(v)f(v)}{F(v)^{\frac{3}{2}}} \leq C\int_{0} ^ {t} \frac{\phi(s)f^{\frac{3}{2}}(s)}{s^{\frac{3}{2}}} \leq C\int_{0} ^ {t} s^{\frac{3\gamma}{2} - \frac{5}{3}} \leq Ct^{\frac{3\gamma}{2} - \frac{2}{3}}.
\end{equation*}
Now, we have that $\Delta \phi(v) \leq C$ by assumption $(A_2)$. Also, the boundary integral term is bounded by a constant and $\frac{\phi(s)f(s)}{F^{\frac{1}{2}}(s)} \geq 1$. As a result, we sum \eqref{eq combined} over a finite collection of balls that cover the free boundary to get that 
\begin{equation*}
\frac{|0 \leq v \leq \phi^{-1}(\theta)|}{\theta} \leq C + C\frac{\phi^{-1}(\theta)^{\frac{3\gamma}{2} - \frac{2}{3}}}{\theta},
\end{equation*}

which implies that

\begin{equation*}
|0 \leq v < \theta| \leq C\phi(\theta) + C\theta^{\frac{3\gamma}{2} - \frac{2}{3}} \leq C\theta^{\frac{1}{3}} + C\theta^{\frac{3\gamma}{2} - \frac{2}{3}}.   
\end{equation*}

Now, notice that $\frac{3\gamma}{2} - \frac{2}{3} \geq \gamma - 1$ for all positive $\gamma$. By iterating the argument $k$ times, it follows that $f(t) \leq Ct^{\gamma_{k}}$, where $\gamma_{k} = \min\{\frac{3\gamma_{k - 1}}{2} - \frac{2}{3}, \frac{1}{3}\}$. Notice that the sequence defined recursively by $a_n = \frac{3a_{n-1}}{2} - \frac{2}{3}$ has no fixed point for $a_0 > 0$ and is increasing. Hence, we have that there exists some $N$ such that $\gamma_{k} = \frac{1}{3}$ for all $k \geq N$. We thus have that $f(t) \leq Ct^{\frac{1}{3}}$.
\end{proof}

\bibliographystyle{acm}
\bibliography{references_mod}
\end{document}